\documentclass{article}

\usepackage{amsthm}
\usepackage{amsmath}
\usepackage{amssymb}
\usepackage[margin=1in]{geometry}
\usepackage{enumerate}
\usepackage{hyperref}
\usepackage{mathrsfs}
\usepackage{color}
\usepackage{bm}
\usepackage{graphicx}
\title{Positive Lyapunov exponent for random perturbations of predominantly expanding multimodal circle maps}
\author{Alex Blumenthal\thanks{School of Mathematics, Georgia Institute of Technology, Atlanta, Georgia, USA. Email: \url{ablumenthal6@gatech.edu}. This material is based upon work supported by the National Science Foundation under Award No. DMS-1604805.}
\and Yun Yang\thanks{Department of Mathematics, Virginia Tech, Blacksburg, VA 24061, USA. Email: \url{yunyang@vt.edu}.
This material is based upon work supported by the National Science Foundation under Award No. DMS-2000167.}
}

\date{\today}

\theoremstyle{theorem}
\newtheorem{thm}{Theorem}
\newtheorem{cor}[thm]{Corollary}
\newtheorem{lem}[thm]{Lemma}
\newtheorem{prop}[thm]{Proposition}

\newtheorem{thmA}{Theorem}

\theoremstyle{definition}

\newtheorem{defn}[thm]{Definition}
\newtheorem{rmk}[thm]{Remark}

\newtheorem{cla}[thm]{Claim}

\newcommand{\E}{\mathbb{E}}

\newcommand{\N}{\mathbb{N}}
\renewcommand{\P}{\mathbb{P}}
\newcommand{\R}{\mathbb{R}}

\newcommand{\Z}{\mathbb{Z}}

\newcommand{\Qc}{\mathcal Q}

\newcommand{\Bc}{\mathcal{B}}
\newcommand{\Fc}{\mathcal{F}}
\newcommand{\Hc}{\mathcal{H}}

\newcommand{\Rc}{\mathcal{R}}
\newcommand{\Gc}{\mathcal{G}}
\newcommand{\Ac}{\mathcal{A}}

\renewcommand{\Hc}{\mathcal{H}}

\newcommand{\Nc}{\mathcal{N}}

\newcommand{\Ic}{\mathcal I}

\renewcommand{\a}{\alpha}
\renewcommand{\b}{\beta}

\newcommand{\e}{\epsilon}

\newcommand{\tX}{\tilde{X}}

\newcommand{\dist}{\operatorname{dist}}

\newcommand{\T}{\mathbb T}

\newcommand{\Leb}{\operatorname{Leb}}

\newcommand{\Len}{\operatorname{Len}}
\newcommand{\Pc}{\mathcal P}

\newcommand{\Supp}{\operatorname{Supp}}

\newcommand{\modone}{\,\, (\text{mod } 1)}
\newcommand{\uo}{{\underline \omega}}

\begin{document}

\maketitle

\begin{abstract}

We study the effects of IID random perturbations of amplitude $\e > 0$ on the asymptotic dynamics of one-parameter families
$\{f_a : S^1 \to S^1, a \in [0,1]\}$ of smooth multimodal maps which are ``predominantly expanding'', i.e., $|f'_a| \gg 1$ away from small neighborhoods of the critical set $\{ f'_a = 0 \}$. We obtain,
for any $\e > 0$, a \emph{checkable, finite-time} criterion on the parameter $a$ for random perturbations of the 
map $f_a$ to exhibit (i) a unique stationary measure, and (ii) 
a positive Lyapunov exponent comparable to $\int_{S^1} \log |f_a'| \, dx$. 
This stands in contrast with the situation for the deterministic dynamics of $f_a$, the chaotic regimes of which
are determined by typically uncheckable, infinite-time conditions. 
Moreover, our 
finite-time criterion depends on only $k \sim \log (\e^{-1})$ iterates of the deterministic dynamics of $f_a$,
which grows quite slowly as $\e \to 0$.

\end{abstract}

\section{Introduction and statement of results}

A fundamental goal in dynamical systems is to determine the asymptotic behavior of various dynamical systems.
Away from the uniformly expanding, Anosov and Axiom A settings, maps can have ``mixed'' dynamical behavior, e.g., 
hyperbolicity on some parts of phase space and contractive behavior on others. On the collection of maps with this `mixed' 
behavior, various dynamical regimes (e.g., asymptotically stable orbits with large basins of attraction versus
more `chaotic' asymptotic behavior)
 can be intermingled, in the space of maps, in an extremely convoluted way. 

These issues are already present in the deceptively simple example of the one-parameter family of quadratic maps $f_a : [0, 1] \to [0, 1], f_a(x) := a x (1 - x)$ for $a \in [0,4]$. Let us agree to say that for a parameter $a \in [0,4]$, the map $f_a$ is \emph{regular} if phase space $[0,1]$ is covered Lebesgue almost-surely by the basins of periodic sinks, while $f_a$ is \emph{chaotic} if it possesses a unique a.c.i.m. with a positive Lyapunov exponent. 
For the family $\{ f_a\}$, it is known (e.g., \cite{L02} and many others) that the parameter space [0,4] is Lebesgue-almost surely partitioned into two sets, $\mathcal A \cup \mathcal B$, with the following properties: 
\begin{itemize}
    \item For all $a \in \mathcal A$, the map $f_a$ is regular, and for all $a \in \mathcal B$, the map $f_a$ is chaotic.
    \item The set $\mathcal A$ is open and dense in $[0,4]$, while $\mathcal B$ has positive Lebesgue measure. In particular, every $a \in \mathcal B$ is the limit point of a sequence $\{ a_n \} \subset \mathcal A$.
\end{itemize}
In particular, the chaotic property is extremely \emph{structurally unstable} with respect to the parameter $a$: any $a \in \Bc$ is the limit point of a sequence $\{ a_n \} \subset \Ac$.

Aside from `exceptional' cases (e.g., $a = 4$),
it is typically impossible to rigorously determine, even with the help of a computer, 
the dynamical regime corresponding to a \emph{given} parameter $a \in [0,4]$, as this determination 
would require infinite-precision knowledge of infinite-length trajectories. 
For the quadratic family and other families of 1D maps with mixed expansion and contraction,
the core issue is the difficulty in ruling out the formation of \emph{sinks of high period}:
even if, for a given $a$, sinks of period $\leq N$ are ruled out for some extremely large $N$, 
one cannot rule out the existence of a sink of period $N + 1$ or greater. Indeed,
the trajectory of a sink of large period may `look' chaotic before the full period
has elapsed.

 Although fewer results are known for higher-dimensional models, one anticipates a similar degree 
 of convoluted intermingling of dynamical regimes: see, e.g., 
 the class of examples now known as Newhouse phenomena \cite{N79}. 
 A somewhat more complete account of coexistence phenomena is 
 available for the famous Chirikov standard map family \cite{C79}, 
 a one-parameter family $\{ F_L, L > 0 \}$ of volume-preserving maps on the torus $\T^2$ exhibiting 
 simultaneously both strong hyperbolicity and elliptic-type behavior on phase space. 
 As the parameter $L$ increases, so too does the proportion of phase space on 
 which $F_L$ is hyperbolic, as well as the ``strength'' of this hyperbolicity.
However, even for large $L$, a small amount of 
 elliptic-type behavior is intermingled with hyperbolic behavior in the parameter space.
 Indeed, for a residual set of large $L$, it is known 
 that elliptic islands for $F_L$ are approximately $L^{-1}$-dense in $\T^2$ (Duarte 1994 \cite{D2}; see also \cite{D1}) , 
 while the set of points with a positive Lyapunov exponent has Hausdorff 
 dimension 2 and is approximately $L^{- 1/3}$-dense in $\T^2$ (Gorodetski 2012 \cite{G12}).
 To the authors' knowledge, it is still not known whether $F_L$ has 
 positive metric entropy (equivalently, a positive Lyapunov exponent on a positive-volume 
 set) for any fixed value of $L$.
  
A similar situation exists for the H\'enon family of diffeomorphisms $f_{a,b}(x,y) :=(1 - a x^2 + y, b x)$ for real parameters $a, b$, introduced
by H\'enon \cite{henon1976two} as a toy model capturing the dynamics of Poincar\'e sections of the Lorenz
model \cite{lorenz1967nature} in certain parameter ranges. Note that the singular limit $b \to 0$ corresponds with the quadratic map family. Of particular interest are the ``classical parameters'' $a = 1.4, b = .3$ at which a wealth of numerical evidence suggests $f_{a,b}$ admits a chaotic strange attractor (see, e.g., the original work \cite{henon1976two}). This remains a major open problem and is likely to be quite difficult: see, e.g., \cite{galias2014structure, galias2015henon} which establishes the existence of parameters close to $(a,b) = (1.4, .3)$ at which the attractor degenerates into periodic sinks $f_{a,b}$ admits sinks. Another known difficulty is the mechanism of 
unfurling of homoclinic tangencies \cite{newhouse1974diffeomorphisms}; for the H\'enon map specifically, see for example \cite{benedicks2018coexistence} and the references therein. 
At present, the existence of strange attractors for $f_{a,b}$ is only known for perturbatively small
values of $b$ \cite{BC2, mora1993abundance}. This work has since been substantially generalized to a framework for establishing existence of \emph{rank one} strange attractors in the work of Wang, Young and others in a variety of contexts, e.g., near limit cycles subjected to 
time-periodic forcing with long period \cite{wang2008toward, WO,OS, LWY}. 
We emphasize that these constructions are quite challenging, and do not explicitly identify
parameters at which the strange attractors exist; instead, a parameterized family of maps is considered, and a nonempty set of parameters (a positive-volume Cantor set) is identified at which a strange attractor exists. 

\subsubsection*{Random perturbations}

The real world is inherently noisy, and so it is natural to consider IID random perturbations of otherwise deterministic dynamics and seek to understand the corresponding asymptotic behavior. For concreteness, let us consider a smooth, deterministic map $f : S^1 \to S^1$ and assume that $|f'| > 2$ on all but a small neighborhood of the critical set $\{ f' = 0\}$ for $f$.

Parametrizing $S^1 \cong [0,1)$ and doing arithmetic ``modulo 1'', at time $n$ we perturb $f$ to the map $f_{\omega_{n-1}}(x) = f(x + \omega_{n-1})$, where  $\omega_0, \omega_1, \cdots$ are IID random variables uniformly distributed in $[- \epsilon, \epsilon]$. Here, the \emph{noise amplitude} $\epsilon > 0$ is a fixed parameter. We will consider the asymptotic dynamics of compositions of the form \[f^n_{\uo} = f_{\omega_{n-1}} \circ \cdots \circ f_{\omega_0}\] given a sample $\uo = (\omega_0, \omega_1, \cdots)$. 

\medskip

%\begin{itemize}
%	\item Asymptotic properties of $f^n_\uo$ should depend only on finitely many iterates
%	\item Assume $f$ is multimodal with $|f'| > 2$ away from quite small neighborhoods of critical points
When $\e \geq 1$, random trajectories $X_n = f^n_\uo(X_0), n \geq 1$ are essentially IID themselves; 
	in this situation it is a straightforward exercise to check (i) uniqueness of the
	stationary measure for the process $(X_n)$ on $S^1$ and (ii) that the Lyapunov exponent $\lambda = \lim_{n \to \infty} 
	\frac1n \log |(f^n_\uo)'(x)|$ exists and is constant for every $x \in S^1$ and a.e. sample $\uo$.
What is more subtle is the situation when $\e \ll 1$, in which case the composition $f^n_\uo$ may
	develop one or more \emph{random sinks}; here, for our purposes, a random sink is a 
	stationary measure for $(X_n)$ with a negative Lyapunov exponent.

	 Random sinks can develop if, for instance, the map $f$ itself has a periodic sink $z \in S^1$. 
	Indeed, it is not hard to check that the sink $z$ persists in the form of a random sink 
	for all $\e > 0$ sufficiently small (see, e.g., 
	Section 3.1 of this paper for a worked example).
	 On the other hand, one anticipates that sinks of $f$ of high period $N$ can be ``destroyed''
		in the presence of a small but sufficient amount of noise, i.e., 
		when $\e \geq \e_N$, where $\e_N \to 0$ as $N \to \infty$.
		As described previously, these high-period sinks are precisely those 
		responsible for the convoluted intermingling of dynamical regimes in one-parameter families of unimodal
		or multimodal maps.

	 In an alternative perspective: given a fixed noise amplitude $\e > 0$, the only
		sinks of $f$ which could possibly persist as random sinks for $(f^n_\uo)$ are
		those of period $\leq k_\e := \max\{ N : \e < \e_N\}$. 
		A crucial point here is that, for a given map $f$, 
		it is virtually always possible to check 
		for sinks of period less than some given value.
		 For these reasons, one anticipates that for a reasonably large class of $f$ as above and a given noise
	amplitude $\e > 0$, it should be possible to determine the asymptotic chaotic regime of the corresponding
	random composition $f^n_\uo$ based on \emph{checkable criteria} involving
	only \emph{finitely many} iterates of the map $f$.
 	
	\medskip
	
	 The present paper is a step in this direction for a model of 
	one-parameter families of multimodal circle maps $f = f_a$ exhibiting strong expansion ($|f_a'| \gg 1$) away from
	a small neighborhood of the critical set $\{ f_a' = 0 \}$.
	We obtain a checkable sufficient criterion on the parameter $a$, involving only finitely
	many iterates of the map $f_a$ (in particular, precluding sinks of low period, as above),
	for deducing asymptotic chaotic behavior for the random composition $f^n_\uo$ 
	when the noise parameter $\e$ is not too small. An appealing feature of these 
	results is that, given $\e > 0$, the criterion involves only approximately $\log(\e^{-1})$ iterates,
	which grows quite slowly as $\e \to 0$.

\subsection{Statement of results}

\subsubsection*{The model}
Let $S^1=\mathbb{R}/ \mathbb{Z}$ be the unit circle, parametrized by the interval $[0,1)$.  We assume throughout that $\psi:S^1\rightarrow \mathbb{R}$ is a $C^2$ function for which the following conditions hold:
\begin{enumerate}
\item[(H1)] the \emph{critical set} $C'_{\psi}=\{\hat{x}\in S^1:\psi'(\hat{x})=0\}$ has finite cardinality, and%and $C''=\{\hat{z}\in S61:\psi''(\hat{z})=0\}$ have finite cardinality;
\item[(H2)] we have $\{ \psi'' = 0 \} \cap C_\psi' = \emptyset$.
\end{enumerate}
%Hypothesis (H2) holds whenever $\psi$ is $C^2$ and $\min \{ |\psi''(\hat x )| : \hat x \in C_\psi' \} > 0$. {\color{red} Mention dependence of $K_1$ when $\psi$ is $C^3$?}

We consider maps of the form
\[
f = f_{L, a} := L \psi + a  \modone \, ,
\]
for $L > 0, a \in [0,1)$, where $\modone : \R \to S^1 \cong \R / \Z$ is the natural projection. Observe that for $L \gg 1$, the map $f$ is strongly expanding away from $C_\psi'$. 

When $\e > 0$ is specified, we write 
$\Omega = \Omega^\e = \big( [- \e, \e]\big) ^{\Z_{\geq 0}}$ for the sample space
for our perturbations. Elements $\uo \in \Omega$ are written
$\uo = (\omega_0, \omega_1, \omega_2, \cdots)$ where $\omega_i \in [- \e, \e], i \geq 0$.
With $\nu^\e$ denoting the uniform distribution on $[-\e, \e]$, we define $\P = \P^\e = (\nu^\e)^{\otimes \Z_{\geq 0}}$
on $\Omega$. We write $\Fc$ for the product $\sigma$-algebra on $\Omega$ and for $n \geq 0$ we write $\Fc_n = \sigma(\omega_0, \omega_1, \cdots, \omega_n) \subset \Fc$.

\medskip

When $f = f_{L, a}$ is specified, we consider random maps of the form 
$f_\omega : S^1 \to S^1, f_\omega(x) := f(x + \omega)$, where it is understood implicitly that the
argument for $f$ is taken $\modone$.
Given a sample $\uo \in \Omega$, we have a corresponding random composition
\[
f^n_\uo := f_{\omega_{n-1}} \circ \cdots \circ f_{\omega_1} \circ f_{\omega_0} 
\]
for $n \geq 1$.

Alternatively, we can view the random maps $f^n_\uo$ as giving rise to a {\it Markov chain} $(X_n)_n$ on $S^1$ defined,
for fixed initial $X_0 \in S^1$, by $X_{n +1} := f_{\omega_n}(X_n)$. The corresponding
Markov transition kernel $P(\cdot, \cdot)$ is defined for $x \in S^1$ and Borel $B \subset S^1$
by
\[
P(x, B) := \P(X_1 \in B | X_0 = x) = \nu^\e \{ \omega \in [- \e, \e] : f_\omega(x) \in B \} \, .
\]
We say that a Borel measure $\mu$ on $S^1$ is \emph{stationary} if 
\[
\mu(B) = \int_{S^1} P(x, B) \, d \mu(x)
\]
for all Borel $B \subset S^1$.

\bigskip

\subsubsection*{Results}

Our results concern the following \emph{checkable, finite-time} criterion $(H3)_{c, k}$ on the dynamics of $f$. 
For now, $c > 0$ and $k \in \N$ are arbitrary.
%\item[ For any $\hat x \in C_\psi'$, we have
\begin{align}
(H3)_{c, k}  \quad \quad \text{ For every } \hat x \in C_\psi' \, , \text{ we have } \quad d(f^l(\hat x) , C_\psi') \geq c \quad \text{ for all } 1 \leq l \leq k \, .
\end{align}

%\begin{defn}\label{M-like}
%We call the maps with condition H1,H2, H3,  ``Misiurewicz-Like" (following the maps in Misiurewicz's paper \cite{M}). 
%\end{defn} 
We now state our results. 

\begin{thmA}\label{thm:ergod}
Let $\beta , c \in  (0,1)$. Let $L > 0$ be sufficiently large, depending on these constants, and assume $f = f_{L, a}$ satisfies $(H3)_{c, k}$ for some arbitrary $k \in \N$. Finally, assume $\e \geq L^{- (2 k + 1)(1-\beta)}$. Then, the random composition $f^n_\uo$ admits a unique (hence ergodic) stationary measure $\mu$ supported on all of $S^1$.
\end{thmA}

%Our main result is the following estimate on the Lypaunov exponent of the random composition $f^n_\uo$.
\begin{thmA}\label{thm:lyapEst}
Let $\beta, c \in (0,1)$. Let $L > 0$ be sufficiently large, depending on these constants, and assume $f = f_{L, a}$ satisfies $(H3)_{c, k}$ for some arbitrary $k \in \N$. Finally, assume $\e \geq L^{-(2k + 1)(1-\beta)+\a}$ where $\a \geq 0$ is arbitrary. Then, the Lyapunov exponent 
\[
\lambda = \lim_n \frac1n \log |(f^n_\uo)' (x)| 
\]
exists and is constant over $x \in S^1$ and $\P$-almost every $\uo \in \Omega$, and satisfies the estimate
\[
\lambda \geq \lambda_0 \log L \, ,
\]
where $\lambda_0 = \lambda_0(\a, k) := \min\{\frac{\a}{k+1}, \frac{1}{10}\}$.
\end{thmA}

Theorems \ref{thm:ergod}, \ref{thm:lyapEst} are approximately sharp, 
in the sense that $(H3)_{c, k}$ is compatible with the formation of 
sinks of period $k + 1$, while such sinks persist under random perturbations $\e \leq C L^{- (2 k + 1)}$ where $C > 0$ is a constant. 
See Proposition \ref{prop:optimal} in Section 3.1 for more information.

A satisfying feature of our results is that, for fixed sufficiently large
$L$ and any given $\e > 0$, to deduce
a large positive exponent for $f = f_{L, a}$ requires validating
condition $(H3)_{c, k}$ with $k = k_\e \approx \log (\e^{-1})$. 
The value of $k_\e$ grows only logarithmically with $\e^{-1}$,
which means that even for quite small $\e > 0$,
Theorems \ref{thm:ergod}, \ref{thm:lyapEst} are
already valid when $(H3)_{c, k}$ is verified for a relatively small value of $k$.

%the criterion
%
%$(H3)_{c, k}$ 
%{\color{red} Notes on "determination problem" for random 
%dynamical systems. Criterion here involves only $k \sim \log (\e^{-1})$
%iterates of dynamics to `determine'.}

%\begin{thm}[Parameters]

%\end{thm}

%{\color{green}Alex, could you please add a paragraph on the comparison of our results and your results with Jinxin and Lai-sang? and the differences of the approach? and maybe compared our results with Lian-Stenlund's result and the methods? Maybe also on the computability of our result. Also compared our result with Wilkinson and Shub's result (random perturbation of maps on sphere).
%}

\subsubsection*{Prior work}

There is a substantial and growing literature on random dynamical systems in low dimensions: we recall below
some of the literature on random dynamical systems
 closest to the present paper, i.e., dealing with random 
 maps having strong expansion mixed with some contraction in phase space.

Lian and Stenlund \cite{LS} consider random perturbations of {\it predominantly expanding} (expanding on most of phase space with a small exceptional set) multimodal maps, more-or-less equivalent to the model in the present paper. They prove that for large enough noise amplitudes, the random system has a unique ergodic stationary measure and a positive Lyapunov exponent. They develop a similar result with smaller noise amplitude assuming a `one time-step' condition on the dynamics, essentially equivalent to $(H3)_{c, 1}$ in our paper. Because we deal with higher-iterate dynamical assumptions, the perturbations we may consider are substantially smaller than those in \cite{LS}.

Stenlund and Sulku \cite{SS} obtain exponential loss of memory for IID compositions $T^n = T_n \circ \cdots \circ T_1$ of random circle maps which are ``expanding on average'': contractive behavior ($\inf |T'| \approx 0$) can appear with positive probability, but the random variable $\inf |T'|$ satisfies a moment condition. The random maps we consider in the present paper \emph{always} have critical points, and so do not satisfy the conditions of \cite{SS}.

In a joint work between the first author, Xue and Young \cite{BXY1, BXY2, blumenthal2020lyapunov}, random perturbations of a model of ``predominantly hyperbolic'' two-dimensional maps are considered. The paper \cite{BXY1} considers a volume-preserving model encompassing the Chirikov standard map, and \cite{BXY2} considers a dissipative (volume-compressing) model of maps having qualitative similarities to the H\'enon maps, while the more recent \cite{blumenthal2020lyapunov} considers systems consisting of arbitrarily many coupled volume-preserving maps. Chaotic properties of the deterministic dynamics in each case are anticipated to hold on large subsets of parameter space, but rigorous verification is largely beyond the scope of current studies. What \cite{BXY1, BXY2, blumenthal2020lyapunov} show is that sufficiently large random perturbations have the effect of ``unlocking'' the hyperbolicity of these systems (positive Lyapunov exponent proportional to the Lebesgue average $\int \log \| dF_x\| \, d x$, estimate of decay of correlations). 
A different but related analysis is carried out in the paper of Ledrappier, Sim\'o, Shub and Wilkinson \cite{LSSW}, which considers IID perturbations applied to a twist map on the sphere. 

Additionally, \cite{BXY1, BXY2} allow smaller random perturbations on assuming a checkable condition involving the first several iterates of the deterministic map, consistent with the finite-time checkable criterion given in the present paper.

To reiterate, the papers \cite{LS, SS, LSSW, BXY1, BXY2, blumenthal2020lyapunov} are emphasized because they deal with random perturbations of maps for which very little is assumed: in these studies, the randomness itself is \emph{leveraged} in a crucial way to `shake loose' hyperbolicity. 
Other works examine random compositions of maps with known `good' asymptotic behavior: by way of example, we mention works on smooth \cite{WSY, BY} and piecewise \cite{Buzzi} expanding maps, maps with a neutral fixed point \cite{AHNV}. 
This also includes works on the problem of stochastic stability: under what conditions
do properties of a given deterministic system persist under small random perturbations? 
There are many works in this important direction, for example, work on 
small random perturbations of Axiom A systems \cite{Ki88, Y86}, unimodal maps
under a (noncheckable) infinite-time condition \cite{BBV, KK, BY1}, and stochastic
stability for H\'enon attractors \cite{BV}. 
We also acknowledge the related problem of \emph{statistical stability}, e.g., how long-time statistics of a dynamical system change within a parametrized family: see, e.g., the 
review \cite{R09}.

The study of deterministic one-dimensional maps with critical points (unimodal or multimodal) has a long history,
a small part of which we recall here.
Naturally we inherit and use some of the ideas developed in this literature. Indeed, our criterion $(H3)_{c, k}$
 is a checkable, finite-time version of various criteria on postcritical orbits of unimodal and multimodal maps
as used by, e.g., Misiurewicz \cite{M}, Jakobson \cite{J}, Collet-Eckmann \cite{CE} and Benedicks and Carleson \cite{BC1}. We note as well the more expository account by Wang and Young \cite{WY}, which we found remarkably helpful in preparing this work.
There are also by now several works attempting to quantify the set of parameters for the quadratic map family at which various dynamical regimes are observed \cite{tucker2009rigorous, luzzatto2006computable, galias2017systematic, golmakani2020rigorous}. 
Also related to our finite-time checkable criteria are frameworks attempting 
to understand dynamical properties at ``finite resolution'' \cite{luzzatto2011finite, elshaarawy2013efficient} or along finite, bounded timescales \cite{blumenthal2020diffusion}.

\subsubsection*{Potential future directions}

A natural future direction is to study small random perturbations of 
the H\'enon map and related models, 
in the hope that one can derive checkable finite-time conditions for 
a positive Lyapunov exponent for the corresponding stationary measure.\footnote{For random systems with absolutely continuous stationary measures, a positivity of Lyapunov exponent implies
existence of ``random strange attractors'' analogous to those for deterministic systems
\cite{ledrappier1988entropy, blumenthal2019equivalence}.} This has been carried out for the standard map and a family of dissipative mappings with ``H\'enon flavor'' using finite-time conditions amounting to three steps of the deterministic dynamics in the previous work
\cite{BXY1, BXY2}; the goal of future work would be to go beyond this and 
derive a succession of stronger finite-time conditions allowing for smaller noise amplitudes. 
 On the other hand, studying Lyapunov exponents for models of this kind this is likely to entail several fundamental challenges not addressed in the present manuscript, e.g., coping with the fact that one must now track
tangent directions as well as the location in phase space.

\subsubsection*{Organization of the paper.}

In Section 2, we derive elementary properties of our model used throughout the paper, 
especially the notion of \emph{bound period} defined in Section 2.2.
In Section 3.1, we discuss the possible formation of sinks of period $k + 1$
 under the condition $(H3)_{c, k}$, verifying the relative sharpness of Theorems \ref{thm:ergod}, \ref{thm:lyapEst}; 
  ergodicity as Theorem \ref{thm:ergod} is then proved in Section 3.2.
The material in Section 3 depends on Section 2 but is otherwise logically
 isolated from the rest of the manuscript.
 The proof of Theorem \ref{thm:lyapEst} occupies the remainder 
 of the paper, Sections 4--6. 
  
%  Section 4 deals with distortion estimates for our model, while in Section 5 
%we formulate and study a certain ``partial averaging process'', which can be thought of as a hybrid between the (averaged) 
%Markov chain $(X_n)$ and the random composition $f^n_\uo$ (for fixed samples $\uo \in \Omega$). Finally, in Section 6
%we carry out the estimation of Lyapunov exponents as in Theorem \ref{thm:lyapEst}.

\subsubsection*{Notation}

%We shall use the following notations. 
\begin{itemize}
	\item Throughout, we parametrize $S^1$ by the half-open interval $[0,1) \cong \R / \Z$. For $s \in \R$, we write $s \modone$ for the projection of $s$ to $[0,1) \cong \R / \Z$ modulo 1.
	\item We define the \emph{lift} $\tilde f : S^1 \to \R$ of $f$ by $\tilde f(x) = L \psi(x) + a$ (i.e., without projecting $\modone$ to $S^1$). We regard $\tilde f$ as a map $\R \to \R$ by extending the domain periodically to all of $\R$. We write $\tilde f_\omega(x) = \tilde f(x + \omega)$. We define the corresponding Markov process $(\tilde X_n)_n$ on $\R$ by setting $\tilde X_{n + 1} = \tilde f_{\omega_n}(\tilde X_n)$.
	\item We write $d( \cdot, \cdot)$ for the metric induced on $S^1$ via the identification with $\R/ \Z \cong [0,1)$. Note that in our parameterization, 
	we have the identity $d(x,y) = \min \{ |x - y|, |x - y \pm 1|\}$. For a set $A \subset S^1$, we write $\Nc_\e(A)$ for the $\e$-neighborhood of $A$ in the metric $d$.
	\item For a point $x \in S^1$ and a set $A \subset S^1$, we define the minimal distance $d(x,A) = \inf_{a \in A} d(x, a)$. For sets $A, B \subset S^1$, we define $d(A, B) = \inf_{a \in A} d(a, B) = \inf_{a \in A, b \in B} d(a, b)$.
	\item Given a set $A \subset S^1$ or $\R$ and $z \in S^1$ or $\R$, we write $A - z = \{ a - z : a \in A \}$ for the set $A$ shifted by $z$.
	\item Given a partition $\zeta$ of $S^1$ (resp. $\R$) and a set $A \subset S^1$ (resp. $A \subset \R$), we write $\zeta|_A$ for the partition on $A$ consisting of atoms of the form $C \cap A, C \in \zeta, C \cap A \neq \emptyset$. 
	\item When it is clear from context, we write $\E$ for the expectation with respect to $\P$.
\end{itemize}

%{\color{red} Markov chain notation: $(X_n)$, transition probabilities $P^n(x, \cdot) = \P(f^n_\uo x \in \cdot)$,
%probability $\P_{X_0}$ when initiated at $X_0$.}

\section{Preliminaries: predominant expansion and bound periods}

\subsubsection*{Bound periods: a heuristic}
Consider the dynamics of a smooth unimodal or multimodal map $f : S^1 \to S^1$.
In the pursuit of finding maps $f$ 
accumulating a positive Lyapunov exponent,
the main obstruction is the formation of sinks, and so
%A crucial idea in 1D dynamics is that to `heal' this damage,
a natural assumption to make is that the postcritical
orbits $f^n \hat x, \hat x \in \{ f' = 0\}, n \geq 1$ remain enough far away from $\{ |f'| \leq 1\}$
so that $|(f^n)'(f \hat x)| \gtrsim e^{n \alpha}$
for some $\a > 0$. 

If, for some $x \in S^1$, the orbit $(f^n x)_n$ reaches a small neighborhood of some $\hat x\in \{ f' = 0\}$
at time $t$, then the subsequent iterates $f^{t + i} x$ will closely shadow $f^i \hat x$ 
for $i \leq p = p(d(f^t x, \hat x))$. The time interval
$[t + 1, t + p]$ is referred to as the \emph{bound period} for $x$ at time $t$. 
As we assumed expansion along the postcritical orbit $(f^i \hat x)_{i \geq 1}$, 
one anticipates that the derivative growth $(f^p)'(f^{t + 1} x)$
accumulated along the bound period will balance out the derivative `damage' due to $f'(f^t x)$ (possibly $\ll 1$ 
when $f^t x, \hat x$ are quite close), so that, for instance,
 $(f^{p + 1})'(f^t x) \sim e^{ (p + 1)\alpha'}$ holds for some $\alpha' < \alpha$.

This is a rough summary of a mechanism by which 1D maps with critical points (unimodal and multimodal)
can accumulate a positive Lyapunov exponent for typical trajectories. For an exposition of this method,
see \cite{WY}. 

\medskip

Our aim in Section 2 is to apply a variation of this idea to our model: the condition $(H3)_{c, k}$
 involves the first $k$ iterates of postcritical trajectories, and so bound periods of length up to
$k$ are available to recover derivative growth. In Section 2.1 we 
carry out some essential preliminaries used in the rest of the paper, and in Section 2.2 we
will discuss bound periods for our random compositions.

\subsection{Preliminaries}

%Section 2.1.1 describes the basic setup for the remainder of the paper-- precise specification
%of the model and dependencies of various constants, while in 2.1.2 we give a preliminary description of the
%predominant expansion exhibited by our model.

\subsubsection{The basic setup}\label{subsubsec:basicSetup}
We fix, below and throughout the paper, a function $\psi: S^1 \to \R$ satisfying (H1) and (H2), as well
as parameters $c \in (0,1), \b \in (0,\frac{1}{100})$ (restricting to $\beta$ in this range incurs no 
loss of generality). 
Moreover, we implicitly fix the parameter $L > 0$, and are allowed to 
take it sufficiently large depending on $c, \b$ and the function $\psi$. 

On rescaling the function $\psi$ in relation to the parameter $L$, we will assume going forward that the following
condition holds in addition to (H1) -- (H2).
\begin{itemize}
	\item[(H4)] We have $\| \psi' \|_{C^0}, \| \psi'' \|_{C^0} \leq 1/10$.
\end{itemize}

\medskip

Separately (i.e., independently of $L$), $k \in \N$ is fixed, and a parameter $a \in [0,1)$ is fixed for which $(H3)_{c, k}$ holds for
the mapping $f = f_{L, a} := L \psi + a \modone$. Finally, we fix a parameter $\e > 0$, on which
constraints (depending on all the previous parameters) will be made as we go along.

\subsubsection{Partition of phase space}

The conditions (H1) -- (H2) imply that there is a constant $K_1 = K_1(\psi) > 0$ with the property that for any $x \in S^1$, 
\begin{align}\label{eq:lowerBoundDer}
|\psi'(x)| \geq K_1 d(x, C_\psi') \, .
\end{align}
We use \eqref{eq:lowerBoundDer} repeatedly, often without mention. For $\eta<0$, we define
\begin{align}
B(\eta) = \{ x \in S^1 : d(x, C_\psi') \leq K_1^{-1}L^{\eta} \} \, .
\end{align}
It is clear that for $x \notin B(\eta)$, we have
$
|f'(x) | \geq L^{\eta + 1} \, ,
$
while $B(\eta)$ is the union of $\# C_\psi'$-intervals of length $\sim L^{\eta}$ each.

\medskip

Define the partition $S^1 = \Gc \cup \Ic \cup \Bc$, where
\[
\Gc = S^1 \setminus B(- \beta) \, , 
\quad \Ic = B(- \beta) \setminus B(- \frac12 - \b) \, ,
\quad \Bc = B(-\frac12 - \b) \, .
\]
We have, then, that 
\[
|f'|_{\Gc}| \geq L^{1 -\b } \, , \quad \text{ and } \quad |f'|_{\Ic}| \geq L^{\frac12-\beta} \, .
\]
Similar estimates apply to $f'_\omega$ on the shifted sets $\Gc_\omega := \Gc - \omega, \Ic_\omega := \Ic - \omega$ for $\omega \in [- \e, \e]$. 

\medskip

Observe that $|f'|_{\Bc}|$ can be arbitrarily small. 
To address this, we subdivide $\Bc = \cup_{l = 1}^k \Bc^l$ in the following way: set
\[
\Bc^k = B(- \frac{k}{2} - \b) \, , 
\]
and for $1 \leq l < k$,
\[
\Bc^l = B(- \frac{l}{2} - \b) \setminus B(- \frac{l + 1}{2} - \b) \, .
\]
Notice that the definition above is consistent with the identification $\Ic = \Bc^0$.
We also use the notation $\Bc^l_\omega := \Bc^l - \omega$ for $\omega \in [- \e, \e]$.
Using \eqref{eq:lowerBoundDer}, one checks that
\[
|f_\omega'|_{\Bc_\omega^l}| \geq L^{- \frac{l - 1}{2} - \b} \quad \text{ for } \quad 1 \leq l < k \, ,
\]
while
on $\Bc_\omega^k$ we have no lower bound on $|f_\omega'|$.

\medskip

The partitions $S^1 = \Gc \cup \Ic \cup \Bc = \Gc \cup \Ic \cup \Bc_1 \cup \cdots \cup \Bc^k$ are used repeatedly
throughout the paper. We will abuse notation and regard these as partitions of $\R$ as well, extended by periodicity
via the parametrization $S^1 \cong [0,1) \cong \R/\Z$.

%{\color{red} Parameters in this section: $L$ taken sufficiently large. No constraints on $k \in \N$ or $\e > 0$. Assume implicitly that $L, a$ are such that $(H3)_{c, k}$ holds.}
%
%
%\subsection{Expansion away from critical set $C_\psi'$}

%\begin{rmk}From \eqref{eq:lowerBoundDer} we have $|f'(x)| \geq L^{\eta + 1}$ for $x \notin B(\eta)$. 
%\end{rmk}

%It says, more or less, 
%that points a distance $\gtrsim L^{-1}$ from the critical set $C_\psi'$ experience strong expansion 
%under the deterministic map $f$.

%More precisely, we use \eqref{eq:lowerBoundDer} to identify families of sets where this expansion fails to varying 
%degrees: 
%. \bigskip

%
%
%
%\bigskip

%{\color{red} Note on parameters: $\psi, c$, then $\b$, then $L, \e$.}

%\subsection{Bound periods for dynamics near critical set $C_\psi'$}

%Condition $(H3)_{c, k}$ says that postcritical orbits of $f$ of length $k$ stay a fixed distance $> c$ away from the
%critical set $C_\psi'$. Trajectories initiated sufficiently close to the critical set will `shadow' these postcritical orbits
%for some amount of time-- this is called a \emph{bound period}. During bound periods, contraction due to proximity
%to the critical set is `healed' by shadowing the postcritical orbit in the strongly expanding region: orbits starting closer
%to the critical set start off more strongly contracted, but shadow the postcritical orbit for longer and so have more
%time to `heal'.

%Let us make this more precise. To start, 

\subsection{Bound periods}

The following lemma confirms that a random orbit $(f^i_\uo x)$, initiated at $x \in \Bc^l, 1 \leq l \leq k$, will
closely shadow a postcritical orbit $(f^i \hat x)$ for $l$ steps, i.e., will have a bound period of length $l$.

\smallskip

In Lemma \ref{lem:boundPeriodCoherence} below we do not assume $(H3)_{c, k}$.

\begin{lem}\label{lem:boundPeriodCoherence}
%Let $c > 0$, $k \in \N$. Let $L$ be sufficiently large in terms of $c, \beta$. Let $f$ be such that $(H3){}_{c, k}$ holds. 
Let $L$ be sufficiently large, and let $k \in \N$ be arbitrary. Assume that 
\begin{align}\label{boundEpsilon}
\e < L^{- \max \{ k-1, \frac12\} - \beta} \, .
\end{align}
%$\e < \min\{L^{- \frac{k}{2} - 2\beta}, L^{- (k - 1) (1 + \beta) - \beta}\}$.
Then, we have the following. Let $1 \leq l \leq k$ and fix an arbitrary 
sample $\uo \in \Omega$.
Let $J_0$ be any connected component of $B(- \frac{l + \b}{2})$ %$\cup_{j = l}^k \Bc^j_{\omega_0}$,
and let $\hat x = C_\psi' \cap J$ be the (unique) critical point contained in $J_0$.

%for any $1 \leq l \leq k$ and any sample $\uo$, we have that
Then, for all $1 \leq i \leq l$ we have that
\[
f_\uo^i \big( 
%\cup_{j = l}^k \Bc^j_{\omega_1}
J_0
 \big) \subset \Nc_{L^{- \b / 2}} (f^i \hat x) \, .
\]
%for all $1 \leq i \leq l $.
\end{lem}
The reason for the upper bound \eqref{boundEpsilon} is that 
 if the perturbation amplitude $\e$ is too large, then 
 $f^i_\uo|_{\Bc^l_{\omega_0}}$ may diverge from $f^i \hat x$ 
 for some $i < k$,
% might come too close to the critical set
% for some $i < k$, 
 thereby spoiling the corresponding bound periods. 

\medskip

From Lemma \ref{lem:boundPeriodCoherence} and noting $\Bc^l \subset B(- \frac{l + \b}{2})$, it is straightforward to check that if $L$ is sufficiently large and
 $f = f_a$ satisfies $(H3)_{c, k}$, then
 $f^i \hat x$ is well inside $\Gc$ for $1 \leq i \leq k$.
It follows that for any $1 \leq l \leq k$ and $x \in \Bc^l_{\omega_0}$, 
we have $f^i_\uo(x) \in \Gc$ for all $1 \leq i \leq l$, and the derivative estimate
\[
|(f^{l}_{\theta \uo})'(f_{\omega_0}x)| \geq L^{l( 1 - \beta)} \, .
\]
Moreover, if $1 \leq l < k$ then we have $|(f_{\omega_0})'(x)| \geq L^{1 - \frac{l+1}{2} - \b}$, hence
\[
|(f^{l + 1}_\uo)'(x)| \geq L^{(l + 1)(\frac12 - \b)} \, .
\]

For the purposes of the preceeding paragraph, it suffices to take $L$ large enough so that $L^{\b} \gg 2 / (c K_1)$; 
note in particular that $L$ does not depend on $k$.

\medskip

\begin{proof}[Proof of Lemma \ref{lem:boundPeriodCoherence}]
%We actually establish a slightly stronger statement: for any connected component $J_0$ of $B(- \frac{l + \b}{2})$ and any sample $\uo$, we have 
%\[
%\dist(f^i_\uo(J_0), C_\psi') > \frac{c}{2} \quad  \forall 1\leq i\leq l.
%\]
%Lemma \ref{lem:boundPeriodCoherence} follows on taking $L$ sufficiently large so that $\frac{c}{2} \gg K_1^{-1} L^{- \b}$, hence $\{ x \in S^1 : d(x, C_\psi') > \frac{c}{2}\} \subset \Gc$.
% when $L$ is sufficiently large so that $k_1^{-1}L^{-2\beta}<\frac c2$, i.e., $L^{2\beta} > 2 (K_1 c)^{-1}$.

In the following proof, the lift $\tilde f : S^1 \to \R$ of $f$ is defined by $\tilde f(x) = L \psi(x) + a$, i.e., leaving out the ``$\modone$''
in the definition of $f$. We extend the domain of $\tilde f$ to all of $\R$ by periodicity.

\medskip

Without loss, we regard $J_0$ as an interval in $\R$. Let $\hat x \in C_\psi' \cap J_0$ be the (unique) critical point in $J_0$. Define $I_0 = \Nc_\e(J_0)$ and inductively set $J_{i+1} = \tilde f(I_i)$, $I_{i + 1} = \Nc_\e(J_{i + 1})$. 
Since $f^i \hat x \in J_i$ for all $i$, it suffices to show $\Len(J_i) \leq L^{- \b/2}$ for all $1 \leq i \leq l$.

%\medskip

%Since $J_i \supset \tilde f^i_\uo (J_0)$ holds for any $i$, 
%it suffices to show that $\dist(J_i, C_\psi') > \frac{c}{2}$ for all $1 \leq i \leq l$. For this, 
%since $\tilde f^i(\hat x) \in J_i$ for each $i$ and $\dist(\tilde f^i(\hat x), C_\psi') > c$ for all $1 \leq i \leq k$, 

To start, decompose $I_0 = I_0^- \cup I_0^+$ where $I_0^- = [\hat x - \e - K_1^{-1} L^{- \frac{l + \b}{2} }, \hat x), I_0^+ = [\hat x, \hat x + \e + K_1^{-1} L^{- \frac{l + \b}{2} }]$.
Noting that the images $\tilde f(I_0^-), \tilde f(I_0^+)$ share the left (resp. right) endpoint $\tilde f(\hat x)$ if $f''(\hat x) > 0$ (resp. $f''(\hat x) < 0$),
% {\color{red} (a left endpoint for each if $f''(\hat x) > 0$ and a right endpoint for each if $f''(\hat x) < 0$)}, 
 we have the estimate
\begin{align*}
\Len(J_1) & \leq \max \{ \tilde f(I_0^+), \tilde f(I_0^-) \} \leq \frac12 L \| \psi'' \|_{C^0} \cdot ( \e + K_1^{-1} L^{- \frac{l + \b}{2} } )^2 \\
& \leq L \max\{ \e, \Len(J_0)\}^2
\end{align*}
using $(H4)$ in the last step. For each $i > 1$, we estimate
\[
\Len(J_i) = \Len(\tilde f(I_{i-1})) \leq L \| \psi' \|_{C^0} \Len(I_{i-1}) \leq L \max \{ \e, \Len(J_{i-1}) \} \, .
\]
by estimating $\Len(I_{i-1}) \leq 2 \e + \Len(J_{i-1}) \leq 3 \max \{ \e, \Len(J_{i-1})\}$ and using $(H4)$.
Bootstrapping, we conclude
\[
\Len(J_i) \leq L^{i-1} \max \{ \e, \Len(J_1)\} \leq \max \{ L^{i-1} \e, L^i \e^2, L^i \Len(J_0)^2 \} \, .
\]
The first two terms are $< L^{- \b}$ by \eqref{boundEpsilon} for all $i \leq k$. For $i \leq l$, the third term is $\leq L^i \cdot 4 K_1^{-2} L^{- l - \b} \leq L^{- \b/2}$. This completes the proof.
\end{proof}

\section{Ergodicity}

In Section 3.1, we prove Proposition \ref{prop:optimal}, which confirms the sharpness of 
Theorems \ref{thm:ergod}, \ref{thm:lyapEst} in the following sense.
To start, condition $(H3)_{c, k}$ for the map $f = f_a$ is compatible with the formation of a sink of period $k + 1$.
For all $\e > 0$ sufficiently small, such sinks persist as random sinks for the random compositions $(f^n_\uo)$, 
i.e., stationary measures for the Markov chain $(X_n)_n$ admitting a negative Lyapunov exponent.
In Proposition \ref{prop:optimal} we make this quantitative by exhibiting a scenario
when $f = f_a$ (i) satisfies $(H3)_{c, k}$; (ii) admits a sink of period $k +1$; and 
(iii) the random composition $(f^n_\uo)$ admits a random sink for all $\e \lesssim L^{- (2 k + 1)}$.
This upper bound for $\e$ approximately matches the upper bound in Theorems \ref{thm:ergod}, \ref{thm:lyapEst},
confirming the view that these results are sharp.

Having established this, in Section 3.2 we proceed with the proof of Theorem \ref{thm:ergod}. We note that in terms of
logical dependence, Section 3 depends on Section 2 and is otherwise independent of the remainder of the paper, Sections 4 -- 6.

\subsection{Sinks}

Let us take on the assumptions made for the map $f = f_{L, a}$ as in Section \ref{subsubsec:basicSetup}, 
except that for Proposition \ref{prop:optimal} we need not assume $(H3)_{c, k}$ holds.
Observe, however, that the hypothesis of Proposition \ref{prop:optimal}, i.e., the 
existence of a sink of period $k + 1$ for $f = f_{L,a}$, is entirely compatible with $(H3)_{c, k}$. 

\begin{prop}\label{prop:optimal}
For all $L$ sufficiently large, depending only on $\psi$, we have the following. 
Let $k \in \N$ be arbitrary, and assume $f = f_{L, a}$ has the property that 
$f^{k + 1} \hat x = \hat x$ for some $\hat x \in C_\psi'$.
Then, for any $\e \leq \frac{1}{49} L^{- (2 k + 1)}$, we have that the 
random composition $f^n_\uo$ admits a stationary measure $\mu$ for which 
\begin{itemize}
\item[(a)] the support  of $\mu$ $\Supp(\mu)$ is contained in a $\frac17 L^{-(k + 1)}$-neighborhood of the orbit $\hat x, f \hat x, \cdots, f^k \hat x$ (in particular, $\Supp \mu \subsetneq S^1$); and
\item[(b)] $\lambda_1(\mu) < 0$.
\end{itemize}
\end{prop}

\begin{proof}
%Write $C = \| \psi'\|_{C^1} = \max \{ \| \psi' \|_{C^0}, \| \psi'' \|_{C^0}, 1\}$.  
We will show that there is a neighborhood $U$ of $\hat x$ such that for a.e. sample $\uo \in \Omega$,
\begin{itemize}
\item[(i)] $f^{k + 1}_\uo(U) \subset U$\, ; and 
\item[(ii)] $|(f^{k + 1}_\uo )'(x)| < \frac12$ for all $x \in U$. 
\end{itemize}
By standard arguments, (i) -- (ii) imply the existence of a stationary measure $\mu$ with Lyapunov exponent $\lambda(\mu) \leq - \frac{\log 2}{k + 1} < 0$ supported in $\{ f^i_\uo x : x \in U, \uo \in \Omega, 0 \leq i \leq k\}$. At the end, we will estimate the size of this support.

Let $\gamma \in (0,1)$ be a constant, to be taken sufficiently small below, and throughout assume that $\e \leq \gamma L^{- (2 k + 1)}$. Set $U$ to be the closed neighborhood of $\hat x$ of radius $r_U = \sqrt \gamma L^{-(k +1)}$. We estimate
\[
\sup_{z \in U} |(f^{i}_\uo)'(z)| \leq \| f' \|_{C^0}^{i-1} \cdot (\e + \sqrt \gamma L^{-(k + 1)}) \cdot \| f'' \|_{C^0} \leq L^{i} \cdot 2 \sqrt \gamma L^{- (k + 1)} \leq 2 \sqrt \gamma  L^{i - (k + 1)}\, ,
\]
having used the elementary bound $|f_\omega(z)| \leq |z + \omega - \hat x| \cdot \| f'' \|_{C^0} \leq L |z + \omega - \hat x|$ for $z$ near $\hat x$. In particular, at $i = k + 1$ we have that 
\begin{align}\label{eq:derivativeBound11}
|(f^{k + 1}_\uo)'|_U| \leq 2 \sqrt{\gamma} \, ,
\end{align}
hence $U$ maps to an interval $f^{k + 1}_\uo(U)$ of length $|f^{k + 1}_\uo(U)| \leq 2 \sqrt \gamma \cdot |U| = 4 \sqrt{\gamma} \cdot r_U$.
 
%  Taking $\gamma = \frac{1}{64}$, 
%%\begin{align}\label{condC}
%%\gamma < \frac{1}{64 } \, ,
%%\end{align}
%we obtain condition (ii) and the bounds $|(f^{i}_\uo)(U)| \leq \frac14 |U|$ for $0 \leq i \leq k + 1$.

Let us now estimate $d(\hat x, f^{k + 1}_\uo(\hat x))$. For simplicity, we pass to the lifts $\tilde f, \tilde f_\omega$: write $\hat x^i = \tilde f^i \hat x, \hat x^i_\uo = \tilde f^i_\uo \hat x$ for $0 \leq i \leq k + 1$. To start,
\[
|\hat x^1 - \hat x^1_\uo| = |\tilde f(\hat x) - \tilde f(\hat x + \omega_0)| \leq \e \cdot \sup_{d(z, \hat x) \leq \e} |f'(z)| \leq \e^2 L \, .
\]
Next, for $i > 0$,
\[
|\hat x^{i + 1} - \hat x^{i + 1}_\uo| = |\tilde f(\hat x^{i}) - \tilde f(\hat x^{i}_\uo + \omega_{i })| \leq  L (\e + |\hat x^i - \hat x^i_\uo|) \, .
\]
Collecting, we obtain
\begin{align*}
d(\hat x, f^{k + 1}_\uo(\hat x)) \leq |\hat x - \hat x^{k + 1}_\uo| & \leq ( L + L^2  + \cdots + L^k) \e + L^{k + 1} \e^2 \\
& \leq 2 L^k \e + L^{k + 1} \e^2 \leq 3 \gamma L^{- (k + 1)} \, , %\frac{3}{64}  L^{-(k+1)} \leq \frac{3}{16}  |U|  \, ,
\end{align*}
here having assumed $L > 2$. We deduce 
\[
d(\hat x, f^{k + 1}_\uo( \hat x)) \leq 3 \sqrt \gamma \cdot r_U \, .
\]
It is easy to check that the same bound $d(\hat x^i, f^i_\uo(\hat x)) \leq 3 \sqrt \gamma \cdot r_U$ holds for any $0 \leq i \leq k$ as well.

To conclude: for (i) it suffices (see \eqref{eq:derivativeBound11}) to take $\gamma \leq 1/16$. For (ii) we estimate as follows for $z \in U$:
\begin{align}\label{eq:estimateImageU}
d(f^{k + 1}_\uo (z) , \hat x) \leq d(\hat x, f^{k + 1}_\uo (\hat x)) + |f^{k + 1}_\uo (U)| \leq 7 \sqrt{\gamma} \cdot r_U \, .
\end{align}
We conclude that $f^{k + 1}_\uo(U) \subset U$ almost surely as long as $\gamma \leq 1/49$.

Finally, to estimate the support of $\mu$ it suffices to repeat the estimate \eqref{eq:estimateImageU} with $f^i_\uo (z), z \in U$
replacing $f^{k + 1}_\uo(z)$. We conclude that $\mu$ is supported in the $7 \sqrt{\gamma} \cdot r_U$-neighborhood
of the periodic sink $\{ f^i \hat x\}_{0 \leq i \leq k}$.
\end{proof}

\subsection{Ergodicity}\label{subsec:stationaryDistro}

As already seen in the proofs of Lemma \ref{lem:boundPeriodCoherence} and
 Proposition \ref{prop:optimal}, the noise amplitude $\e$ is amplified by
 the strong expansion $L \gg 1$ exhibited by $f = f_{L, a}$. Each of these results
 depended on the noise being \emph{small enough} to control this
amplification. Quite to the contrary, in Section 3.2
we will \emph{take advantage of} this amplification to show that our process $(X_n)$ explores all of phase
space $S^1$ with some positive probability.  The amplification of noise by expansion 
 is a core motif in this paper, one which we will return to in Sections 5 -- 6.

\medskip

Before proceeding to the proof of Theorem \ref{thm:ergod}, let us establish
the setting and a brief reduction. Throughout, we assume the setup for $f = f_{L, a}$
in Section \ref{subsubsec:basicSetup}, including $(H3)_{c, k}$.

\medskip

\medskip

 {\it Reductions. } We first argue that without loss of generality, in
the hypotheses of Theorem \ref{thm:ergod} we may assume
that $\e, k$ are such that the upper bound in \eqref{boundEpsilon} is satisfied, 
so that Lemma \ref{lem:boundPeriodCoherence} applies.
To justify this, consider the following alternative cases: (a) $L^{- (k - 1)} \leq \e < L^{-1}$; (b) $L^{-1} \leq \e <  L^{-1/2}$;
and (c) $\e \geq L^{-1/2}$. For (a), let $k' \in \N$ be such that $L^{- k'} \leq \e < L^{-(k' - 1)}$. Clearly
$k' < k$, hence $(H3)_{c, k}$ implies $(H3)_{c, k'}$, while $\e \geq L^{- k'} \geq L^{- (2 k' + 1)(1 - \b) + \b}$. So, 
it makes no difference to replace $k$ with $k'$ and proceed as before. In case (b), we can replace $k$ with $1$
and proceed as before. Finally, Theorem \ref{thm:ergod} in case (c) is a simple exercise left to
the reader-- see also Theorem 1 in \cite{LS}, where ergodicity as in Theorem \ref{thm:ergod} is proved
for $\e \gtrsim L^{-1}$ for a very similar model of multimodal circle maps.

\smallskip

In addition, on shrinking the parameter $\b$ we will assume the slightly stronger hypothesis
\[
\epsilon \geq L^{- (2 k + 1)(1 - \b) + \b}
\]
on the noise parameter $\e$. In relation to Theorem \ref{thm:ergod}, this incurs
no loss of generality.

\medskip

{\it Notation. } Given an initial $X_0 \in S^1$, we write
$X_n = f^n_\uo(X_0)$ for the Markov chain evaluated at the sample
$\uo \in \Omega$ (notation as in Section 1.1).
We write $\P_{X_0}$ for the law of $X_n$ conditioned on the value of $X_0 \in S^1$.
Moreover, for $n, m \geq 0$, 
random variables $Z_1, Z_2, \cdots, Z_m : \Omega \to \R$, and $X_0 \in S^1$,
we write 
\[
{P^n(X_0, { \cdot } | \{ Z_j, 1 \leq j \leq m \}) = \P_{X_0} (X_n \in \cdot | \sigma(Z_1, \cdots, Z_m))
}\]
for the law of $X_n$
conditioned on $\sigma(Z_1, Z_2, \cdots, Z_m)$.

%Let us now prove the following slight strengthening of Theorem \ref{thm:ergod}.
%\begin{thm}\label{thm:ergodicity2}
%Let $\b \in (0,1/10), c > 0$ and $\psi$ satisfying (H1),(H2) be fixed. Let $L$ be sufficiently large
%in terms of these parameters. Let now $k \in \N$ and $a \in [0,1)$ be such that $f = f_{L, a}$
%satisfies $(H3)_{c, k}$, and assume $\e \geq L^{- (2 k + 1) (1 - \b) + \b}$. 
%
%Then, there is a unique (hence ergodic) stationary measure $\mu$ for the random composition $(f^n_\uo)$
%which is supported on all of $S^1$.
%\end{thm}

\bigskip

With the setup and reduction established, we now turn to the proof of Theorem \ref{thm:ergod}. 
We break this up into two parts, Propositions \ref{prop:bkToSpread} and \ref{prop:allToBk} below.
\begin{prop}\label{prop:bkToSpread}
%Assume the hypotheses of Theorem \ref{thm:ergodicity2}. Then, t
There exist $N \in \N, c > 0$ with the
property that for any sample $\uo$ and any $X_0 \in \Bc_{\omega_0}^k$, we have that $P^N(X_0, { \cdot } | \{\omega_i, 0 \leq i \leq N, i \neq 1\})\geq c \Leb(\cdot)$.
\end{prop}
What this means is that random trajectories initiated in $\Bc^k$ reach all of $S^1$ with some positive probability.
Note that in Proposition \ref{prop:bkToSpread}, we randomize only in $\omega_1$. One reason is that since 
$X_0 \in \Bc^k_{\omega_0}$, we have that $X_1, X_2, \cdots, X_k$ experience a bound period of length $k$, 
and so $\omega_1$ is the only perturbation which experiences the full $k$ steps of expansion
guaranteed by Lemma \ref{lem:boundPeriodCoherence}. Meanwhile,
it is technically more convenient to work with one perturbation $\omega_i$ at a time. 

\medskip

By Proposition \ref{prop:bkToSpread}, it suffices to check that 
almost every trajectory enters $\Bc^k$ after a finite time. Define the stopping time
\[
T := \min \{ i \geq 0 : X_i \in \Bc^k_{\omega_i} \} \, .
\]

\begin{prop}\label{prop:allToBk}
Assume the hypotheses of Theorem \ref{thm:ergod}. Then, there exists $\hat N \in \N$ such
that for any $X_0 \in S^1$,
we have $\P_{X_0}(T \leq \hat N) > 0$.
\end{prop}

\begin{proof}[Proof of Theorem \ref{thm:ergod} assuming Propositions \ref{prop:bkToSpread}, \ref{prop:allToBk}]
Observe that ergodic measures $\mu$ (1) 
exist by a standard tightness argument, 
and (2) automatically inherit absolute continuity w.r.t. Lebesgue on $S^1$ 
from the same property for our random perturbations $\omega_i, i \geq 0$. 
So, to conclude uniqueness it suffices to check
that for all $X_0 \in S^1$, $P^M(X_0, \cdot)$ is supported on all of $S^1$ (i.e., assigns positive mass to all open intervals)
for some $M = M(X_0) \in \N$.
For more details, see, e.g., the characterization 
of ergodicity for stationary measures of random dynamical systems 
in Lemma 2.4 on pg. 19 of \cite{kifer2012ergodic}.

To complete the proof, fix $X_0 \in S^1$ and let $n \leq \hat N$ be such that $\P_{X_0}(T = n) > 0$. Then, for any interval $J \subset S^1$ with nonempty interior,
\begin{align*}
P^{n + N }(X_0, J) 
& = \E \bigg(  P^N\big(X_{n}, J \big| \{\omega_i\}_{0 \leq i \leq n+N, i \neq n+1} \big) \bigg) \\
& \geq \E \bigg( \chi_{T = n} \cdot P^N\big(X_{n}, J \big| \{\omega_i\}_{0 \leq i \leq n+N, i \neq n+1} \big) \bigg)\\
%& \geq \int d (\nu^\e)^{\otimes (n + N)} (\omega_0, \cdots, \omega_{n-1}, \omega_{n + 1}, \cdots, \omega_{n + N})
%\chi_{T = n-1} \, P^N(X_{n-1}, I | \{ \omega_i, n-1 \leq i \leq n + N, i \neq n \}) \\
 & \geq \E \bigg(  \chi_{T = n} \cdot c \Leb(I)  \bigg)  = c \cdot \P_{X_0}(T = n ) \cdot \Leb(I) > 0 \, . 
\end{align*}
Here, $\E_{X_0}$ refers to the 
expectation conditioned on the value of $X_0$. 

This completes the proof. It remains to check Propositions \ref{prop:bkToSpread}, \ref{prop:allToBk}.
\end{proof}

%In \S \ref{ergodPrelim} below, we provide some constructions used in the proof. \S \ref{provePropa} proves
%Proposition \ref{prop:bkToSpread} and \S \ref{provePropb} proves Proposition \ref{prop:allToBk}

% then we can replace $k$ with an alternative value $k' \leq k$ for which \eqref{boundEpsilon} does hold. If this cannot be done, i.e., $\e \geq L^{- \frac12}$, then ergodicity as in 
%Theorem \ref{thm:ergod} is a simple exercise left to the reader ({\color{red} see also Theorem ?? in \cite{??}}).

\medskip

{\it In the remainder of Section 3, we prove Propositions \ref{prop:allToBk}, \ref{prop:bkToSpread}, in that order. 
With the above setup assumed, we hereafter fix $\e \in [L^{- (2 k + 1)(1 - \b) + \b}, L^{- \max\{ k -1, \frac12\}}]$.}

\subsubsection{Constructions and a preliminary Lemma}

%Before giving the proof of Proposition \ref{prop:allToBk}, we first give a useful preliminary Lemma.

Define $\Rc$ to be the partition of $S^1$ into the connected components of the sets $\Gc, \Ic = \Bc^0, \Bc^1, \cdots, \Bc^k$.
For $\omega \in [- \e, \e]$, let $\Rc_\omega$ denote the partition into atoms of the form $\a - \omega, \a \in \Rc$.
Extending by periodicity, we regard $\Rc, \Rc_\omega$ as partitions on $\R$ as well. Given an interval $J \subset \R$, let us write $\Rc|_J = \{ \a \cap J : \a \in \Rc\}$. For $\omega \in [- \e, \e]$, 
the partition $\Rc_\omega|_J$ of $J$ is defined analogously.

\begin{lem}\label{lem:atomCut}
Assume $\bar J \subset \R$ is an interval with $|\bar J| < L^{- \beta}$. Let $J$ be the longest atom of $\Rc|_{\bar J}$. Then, $|J| \geq \kappa |\bar J|$, where $\kappa = \min \{ \frac15, K_1^{-1}\}$.
\end{lem}
\begin{proof}[Proof of Lemma \ref{lem:atomCut}]
Some notation for this proof: given $\hat x \in \{ \tilde f' = 0 \} \subset \R$ and $0 \leq l \leq k$, define 
$\Bc^{l, +}(\hat x)$ to be the connected component of $\Bc^l$ to the immediate right of $\hat x$, and $\Bc^{l, -}(\hat x)$
to be the connected component to the immediate left. Let us write $\Bc(\hat x)$ for the component of $\Bc$ containing $\hat x$.
%Similarly, we write $\Bc^+(\hat x)$ (resp. $\Bc^-(\hat x)$) for the component of $\Bc$ to the immediate right (resp. left) of $\hat x$.

\smallskip

If $\Rc|_{\bar J}$ has only one or two atoms of positive length, then $|J| \geq \frac15 |\bar J|$ holds trivially.
Hereafter we assume $\Rc|_{\bar J}$ consists of three or more atoms of positive length. In particular,
$\bar J$ contains a connected component of $\Bc^l$ for some $0 \leq l \leq k$, since $|\bar J| < L^{- \beta}$ was assumed.
Let $\hat x \in \{ \tilde f' =0 \}$ be the nearest critical point to $\bar J$.
%Letting $\hat x \in \{ \tilde f' = 0 \}$ be the nearest critical point to $\bar J$, 
%Define
%$\bar J^- = \bar J \cap (- \infty, \hat x], \bar J^+ = \bar J \cap [\hat x, \infty)$. WLOG 
%let us assume $|\bar J^+| \geq |\bar J^-|$, so that $|\bar J^+| \geq \frac12 |\bar J|$.
%\begin{cla}
%There is an interval $J' \subset \bar J^+$ with $|J'| = |\bar J^+|$.
%\end{cla}
%Assuming the claim, we can without loss assume that $J \subset \bar J^+$ holds.

Define
\[
l_1 = \min \{ 0 \leq l \leq k : \bar J \text{ contains a component of } \Bc^l\} \, .
\]
There are two cases: (i) $J \subset \Bc^{l_1}$, in which case $J = \Bc^{l_1, \pm}(\hat x)$ for some choice of $\pm$,
or (ii) $J \cap \Bc^{l_1} = \emptyset$.

For case (i), assume first that $l_1 = 0$. WLOG we assume $J = \Bc^{0, +}(\hat x)$. Note that $\bar J \cap \Gc$ consists
of at most two components, hence $|\bar J \cap \Gc| \leq 2 |J|$, while $\bar J \cap \Bc$ has one component, hence
$|\bar J \cap \Bc| \leq 2 K_1^{-1} L^{- \frac12 - \b} \leq 2 L^{- \frac12} |J|$. Finally, $\bar J \cap \Ic$ has at most two components, 
and so $|\bar J \cap \Ic| \leq 2 |J|$. In total,
\[
|\bar J| \leq |\bar J \cap \Gc| + |\bar J \cap \Ic| + |\bar J \cap \Bc| \leq (4 + 2L^{- \frac12}) |J| \leq 5 |J| \, .
\]

Assuming now that $l_1 > 0$, WLOG we have $J = \Bc^{l_1, +}(\hat x)$. Moreover,
 $\bar J \subset \cup_{i = l_1 - 1}^k \Bc^l$; otherwise, $\bar J$ would contain an intact
component of $\Bc^{l_1-1}$, a contradiction. As before,
$\bar J \cap \Bc^{l_1 - 1}$ has at most two components, each of length $\leq |J|$, while
$\bar J \cap \cup_{l_1 + 1}^k \Bc^{l}$ has at most one component of length 
\[
\leq 2 K_1^{-1} L^{- \frac{l_1 + 1}{2} - \b} %\sum_{i = l_1 + 1}^k L^{- \frac{i}{2} - \b} 
\leq 2 L^{- \frac12} |J| \ll |J| \, ,
%\leq 3 K_1^{-1} L^{- \frac{l_1 + 1}{2} - \b} \leq 3 L^{- \frac12} |J| \ll |J| \, .
\]
unless $l_1 = k$, in which case we can ignore this contribution. As before, we conclude $|\bar J| \leq 3 |J|$.

\medskip

For case (ii), if $l_1 = 0$, then $J \subset \Gc$. Note 
$\bar J$ does contains some atom $\Bc^{0, \pm}(\hat x)$, hence $|J| \geq K_1^{-1} L^{- \b} > K_1^{-1} |\bar J|$,
having assumed in Lemma \ref{lem:atomCut} that $|\bar J| < L^{- \b}$.
%From here, the same arguments as in case (i), $l_1 = 0$ apply.
%Like in case (i), $\bar J \cap \Gc$ has at most two components of length $\leq |J|$ each,
%while $\bar J \cap \Bc$ has at most one component of length $2 L^{- \frac12} |J|$.
%Finally, $\bar J \cap \Ic$ has at most two components of length $\leq |J|$ each. In total
%we estimate $|\bar J| \leq 5 |J|$ as before.

If $l_1 > 0$, then likewise it is not hard to show that $J \subset \Bc^{l_1 - 1}$. As before,
$\bar J$ contains some $\Bc^{l_1, \pm}(\hat x)$
and so $|J| \geq K_1^{-1} L^{- \frac{l_1}{2} - \b}$ holds. One now repeats the same arguments
as for case (i), $l_1 > 0$.
\end{proof}

%\begin{proof}[Proof of Proposition \ref{prop:allToBk}]

\subsubsection{Proof of Proposition \ref{prop:allToBk}}

To prove Proposition \ref{prop:allToBk}, we introduce the \emph{random interval process} $(J_i)_{i \geq 0}$ 
of subintervals of $\R$, defined as follows.
Fix $X_0 \in S^1$. 
To start, $J_0 := X_0 + [- \e, \e]$, regarded as an interval in $\R$. 
We set $\bar J_1 := \tilde f(J_0)$ and define 
$J_1$ to be the longest atom of $\Rc_{\omega_1}|_{J_1}$; if more than one 
atom has maximal length, then select $J_1$ to be the rightmost one.
Inductively, given $J_0, \cdots, J_i$, define $\bar J_{i + 1} := \tilde f_{\omega_i}(J_i)$ 
and $J_{i + 1}$ to be the longest atom of  $\Rc_{\omega_{i + 1}}|_{\bar J_{i +1 }}$, 
with the same rule if there is a tie for longest atom.

We terminate the process $(J_i)_i$ at the stopping time $\sigma := \min \{ \sigma_1, \sigma_2\}$, where
\begin{gather*}
\sigma_1 := \min \{ i : |\bar J_i| > L^{- \b}\} \, , \quad \sigma_2 := \min \{ i : J_i \subset \Bc^k_{\omega_i}\} \, .
\end{gather*}

\begin{lem}\label{lem:boundSigma}
There exists $\hat N = \hat N(k, \beta) \in \N$ for which $\P_{X_0} (\sigma \leq \hat N-1) > 0$ holds.
\end{lem}
\begin{proof}[Proof of Proposition \ref{prop:allToBk} assuming Lemma \ref{lem:boundSigma}]
Observe that for each $i \geq 0$, 
\[
\bar J_i \subset \tilde f^{n-1}_{\theta \uo} \circ \tilde f \big( \Nc_\e(X_0) \big) \, ,
\]
hence the projection $ \bar J_i \modone $ of $\bar J_i$ to $S^1$ is a subset of the support of 
the measure $\P_{X_0}(X_i \in \cdot | \{ \omega_i \}_{i \neq 0})$. 

On the event $\sigma = \sigma_1 = m$ for some $m \geq 0$, it is not hard to see that 
$|\tilde f_{\omega_m}(\bar J_m)| \gg 1$ (see Section 2), hence on the event $\{ \sigma = \sigma_1\}$ 
we have $T \leq \sigma_1 + 1$.
Meanwhile, $T \leq \sigma_2$ holds unconditionally (note $\tilde X_m \in \Bc^k_{\omega_m}$ iff $X_m \in \Bc^k_{\omega_m}$), hence 
\[
T \leq \sigma + 1
\]
holds almost surely.
%\[
%\P_{X_0} (\sigma \leq  \leq \P_{X_0}(T \leq m)
%\]
%for all $m \in \N$.
%
%Since $\P_{X_0}(\sigma \leq \hat N-1) > 0$, we conclude $\P_{X_0}(T \leq \hat N) > 0$ as desired. 

To complete
the proof of Proposition \ref{prop:allToBk}, it remains
to prove Lemma \ref{lem:boundSigma}.
\end{proof}

\begin{proof}[Proof of Lemma \ref{lem:boundSigma}]
We will show that conditioned on $\{ \sigma_2 > \hat N\}$, we have $\sigma_1 \leq \hat N$.

%We will show that the event $\{ \sigma > \hat N\} = \{ \sigma_2 > \hat N, \sigma_1 > \hat N\}$ is empty by contradiction. 
Define $t_1 = \min \{ t : J_t \subset \Bc_{\omega_t}\}$ and let $p_1\in \{ 1 , \cdots, k-1\}$ be
such that $J_{t_1} \subset \Bc^{p_1}_{\omega_{t_1}}$.
Inductively, for $j > 1$ set
\[
t_j = \min\{ t > t_{j-1} : J_t \subset \Bc_{\omega_t}\}
\]
and let $p_j$ be such that $J_{t_j} \subset \Bc^{p_j}_{\omega_{t_j}}$.
We let $q \geq 0$ be such that $t_q \leq \hat N < t_{q + 1}$ (note $q = 0$ is allowed).

At time $t_j$, the interval process $J_{t_j}$ is said to initiate a \emph{bound period}
of length $p_j$; that is, $J_{t_j + 1}, \cdots, J_{t_j + p_j}$ shadow some postcritical
orbit in the sense of Lemma \ref{lem:boundPeriodCoherence}. In particular,
$t_j + p_j + 1 \leq t_{j + 1}$ for all $j$. For $t_j + p_j + 1 \leq t \leq t_{j + 1}$, we say that
the interval $J_t$ is \emph{free}.

When $t$ is free, expansion on $\Gc \cup \Ic$ (see Section 2) and Lemma \ref{lem:atomCut}
imply 
\begin{align}\label{eq:growFreeIntervals}
|J_{t + 1}| \geq \kappa |\bar J_{t + 1}| \geq \kappa L^{\frac12 - \b} |J_t| \, ,
\end{align}
while along bound periods (having conditioned on $\{ \sigma_2 > \hat N\}$, it follows that $p_j < k$ for all $j \leq q$) we have
\begin{align}\label{eq:growBoundPeriodInterval}
J_{t_j + p_j + 1} \geq \kappa |\bar J_{t_j + p_j + 1}| \geq \kappa L^{(\frac12 - \b)(p_j + 1)} |J_{t_j}|
\end{align}
since, by Lemma \ref{lem:boundPeriodCoherence},
we have $\bar J_{t_j + p_j + 1} = \tilde f^{p_j + 1}_{\theta^{t_j} \uo} J_{t_j}$ 
(i.e., no cutting can occur during a bound period).
We obtain
that when $J_t$ is free, we have
\[
|\bar J_{t}| \geq \bigg( \kappa L^{\frac12 - \b} \bigg)^t \cdot 2 \e \geq L^{t (\frac12 - 2 \b)} \cdot 2 \e \, .
\]
when $L$ is sufficiently large. Since, for any $t$, the interval $J_{t'}$ is free
for at least one $t' \in \{ t, \cdots, t + k\}$, and  
$\e \geq L^{- (2 k + 1)(1 - \b) + \b}$ was assumed, it follows that
$\sigma_1 \leq \hat N$, where $\hat N = \hat N(k, \beta)$ depends on $k, \beta$ alone.
\end{proof}
%Define 
%\begin{align}\label{eq:defineN11}
%\hat N = k + \bigg\lceil \frac{2k + 1}{\frac12 - 2 \b} \bigg\rceil \, .
%\end{align}
%If $J_{\hat N-1}$ is free, then 
%\[
%L^{- \b} \geq \bigg( \kappa L^{\frac12 - \b} \bigg)^{\hat N-1} \cdot 2 \e \geq L^{(\hat N-1) (\frac12 - 2 \b)} \cdot L^{- (2 k + 1)(1 - \b) + \b}
%\]
%leads to a contradiction, while if $J_{\hat N-1}$ is bound, then $J_{\hat N - 1- p}$ is free for some $p \in \{1, \cdots, k-1\}$, similarly leading to a contradiction.

\subsubsection{Proof of Proposition \ref{prop:bkToSpread}}\label{provePropa}

Assume $X_0 \in \Bc^k_{\omega_0}$. We form what is essentially the same 
interval process as before, starting
now with the interval 
\[
J_1 := X_1 + [- \e, \e] \, ,
\]
again regarded as a subset of $\R$, and taking $\bar J_2 := \tilde f(J_1)$, and $J_2 \in \Rc|_{\bar J_2}$
the longest atom. The intervals $J_3, J_4, \cdots$ are defined the same as before.

As in the proof of Lemma \ref{lem:boundSigma}, no cutting occurs during the initial bound period of length $k$, hence 
$\bar J_{k + 1} = \tilde f_{\theta^2 \uo}^{k-1} \circ \tilde f (\Nc_\e(X_1))$. 
By Lemma \ref{lem:boundPeriodCoherence} and Lemma \ref{lem:atomCut}, this implies
\[
|J_{k + 1}| \geq \kappa |\bar J_{k + 1}| \geq L^{- (k + 1)(1 - \b) + \b/2} \, ,
\]
perhaps taking $L$ sufficiently large (independently of $k$).

With $t_1 = 0, p_1 = k$ and $t_j, p_j, j \geq 2$ defined as in the proof of Lemma \ref{lem:boundSigma},
note that if $p_j < k$ then \eqref{eq:growBoundPeriodInterval} holds, while if $t$ is free we have that \eqref{eq:growFreeIntervals} holds. It remains to check that some interval growth occurs when $p_j = k$; we do so below.

\begin{lem}
Assume $L$ is sufficiently large, depending on $\beta$. Let $J \subset \Bc^k_{\omega_0}$ be an interval for which $|J| \geq L^{- (k + 1)(1 - \b) + \gamma}$ for some constant $\gamma > \beta / 2$. Then,
$|\tilde f^{k + 1}_\uo(J)| \geq L^{- (k+1)(1 - \b) + \frac32 \gamma}$.
\end{lem}
\begin{proof}
	It suffices to estimate the length of $\tilde f_{\omega_0}(J)$. For this, let us subdivide $J = J^+ \cup J^-$, 
	where $J^+$ is to the right of the critical point and $J^-$ to the left. WLOG let $J^+$ be the longer of the two intervals,
	so $|J^+| \geq \frac12 |J|$ holds.
	
	Writing $J^+ = [\hat x - \omega_0, \hat x - \omega_0 + b^+], b^+ > 0$ (noting $b^+ \geq \frac12 |J|$), we have
	\[
		(*) = |\tilde f_{\omega_0}(J^+)| = \int_{\hat x}^{\hat x + b^+} | \tilde f'(x)| dx \geq K_1 \int_{0}^{b^+} x \, dx 
		= \frac12 K_1 (b^+)^2 \geq \frac18 K_1 |J|^2
	\]
	Plugging in the lower bound for $|J|$ gives $(*) \geq \frac18 K_1 L^{- 2 (k + 1)(1 - \b) + 2 \gamma}\geq L^{- 2 (k + 1)(1 - \b) + \frac32 \gamma}$. From here, using Lemma \ref{lem:boundPeriodCoherence} we estimate
	\[
	| \tilde f^{k + 1}_\uo(J)| \geq | \tilde f^{k + 1}_\uo (J^+)| \geq L^{- (k + 1) (1 - \b) + \frac32 \gamma} \, . \qedhere
	\]
\end{proof}

Proposition \ref{prop:bkToSpread} now follows from a similar argument to that for Lemma \ref{lem:boundSigma}, 
where $N = N(k, \beta) \in \N$ and the constant $c > 0$ depends on $N$ as well as $L$.
Details are left to the reader.

\section{Itineraries and distortion}

For the remainder of the paper we turn our attention to the proof of Theorem \ref{thm:lyapEst}. 
In essence, this proof will be an elaboration on the idea, used heavily in Section 3.2, 
 that the predominant expansion of $f = f_{L, a}$ has the effect of amplifying the noise $\e$.
 On the other hand, in Section 3.2 and the proof of ergodicity as in Theorem \ref{thm:ergod}, 
 we were able to avoid exerting any precise control on the densities of the conditional laws 
 $P^n(X_0, \cdot | \{ \omega_i, i \neq 0\})$. For our purposes in Section 6, however, 
 we will need some control on these densities, which amounts to controlling distortion of
 the random compositions $f^n_\uo$.
 
 Our objective in Section 4, then, is to establish some control on the distortion of $f^n_\uo$.
 As is typical of systems exhibiting nonuniform expansion, distortion of $f^n_\uo$ for some $n \geq 1$ 
 can only be controlled along sufficiently small intervals $J \subset S^1$ (see, e.g., \cite{WY}). 
 Establishing just how small these intervals need to be is a crucial component of our argument.
 
 In Section 4.1, we formulate \emph{itineraries} for the random dynamics of $f^n_\uo$, a form
 of symbolic dynamics for the trajectories of $f^n_\uo$ with the property (checked in Section 4.2)
 that the distortion of $f^n_\uo$ can be controlled along subintervals with the same itinerary (symbolic sequence)
 out to time $n-1$. 
 
The preceding paragraphs apply equally well to deterministic as well as random compositions
of interval maps-- indeed, the assignment of itineraries to control distortion is an old idea (see the references
in \cite{WY} for more information). Something to keep in mind, however, is that since the condition
$(H3)_{c, k}$ only guarantees bound periods up to length $k$, we lose control of the dynamics
of $f^n_\uo$ upon the first visit to the `worst possible' neighborhood $\Bc^k$ of $\{ f' = 0 \}$.
Thus the itinerary subdivision procedure and and resulting distortion estimates we obtain
 below are only valid up until this first visit to $\Bc^k$. This issue will be addressed in Section 5.

\subsection{Itineraries}

\emph{Throughout, in addition to the preparations in 
Section \ref{subsubsec:basicSetup}, we assume the parameter $\e$ satisfies 
the upper bound \eqref{boundEpsilon}, so that Lemma 
\ref{lem:boundPeriodCoherence} holds. No
lower bound on $\e$ is assumed.
}

\subsubsection*{(A) Partition construction.}

\noindent To start, we define the partition $\Pc$ of $S^1$ as follows. Recall the notation $\Bc^0 = \Ic$.%Below, we index the critical points $\hat x_1, \cdots, \hat x_M \in S^1$ for $\psi$, $M = \# C_\psi'$.
\begin{itemize}
	\item $\Pc|_{\Gc}$ is the partition of $\Gc$ into connected components.% $\Gc_1, \cdots, \Gc_M$, where we index so that for each $m$, $\Gc_m$ is situated between $\hat x_m$ and $\hat x_{m + 1}$ (writing $\hat x_{M + 1} = \hat x_1$ for consistency).
	\item To define $\Pc|_{\Bc^l}, 0 \leq l < k$, start by cutting $\Bc^l$ into connected components. For each such component $J$, $\Pc|_{J}$ is defined as any
	partition of $J$ into intervals of length 
	\[
	\in [ (l + 1)^{-2} L^{- \frac{l + 3}{2} - \b}, 2  (l + 1)^{-2} L^{- \frac{l + 3}{2} - \b} ] \, .
	\]	\item $\Pc|_{\Bc^k}$ is the partition of $\Bc^k$ into connected components.
\end{itemize}
\noindent We write $\Pc_\omega$ for the partition of $S^1$ with atoms of the form $C - \omega, C \in \Pc$. Abusing notation somewhat, we regard $\Pc, \Pc_\omega$ as partitions of $\R$, extended by periodicity.

\begin{defn}
For a bounded, connected interval $I \subset S^1$ (or $\subset \R$) which is not a singleton, we define the partition $\Pc_\omega(I)$ of $I$ as follows. To start, form $ \Pc_\omega|_{I} = \{ J \cap I :  J \in  \Pc_\omega, J \cap I \neq \emptyset\}$, and write $J_1, J_1, \cdots, J_{N}$ for the non-singleton atoms of this partition in increasing order from left to right (note that $N = 1$ is possible).
\begin{itemize}
	\item If $N = 1,2$ or $3$, then set $\Pc_\omega(I) := \{ I \}$.
	\item If $N \geq 4$, then set $\Pc_\omega(I) = \{ J_1 \cup J_2, J_3, J_4, \cdots, J_{N - 2}, J_{N-1} \cup J_N\}$.
\end{itemize}
\end{defn}

We define the \emph{bound period} $p(I)$ of an interval $I$ as follows. First,  $p : S^1 \to \{ 0,\cdots, k\}$ (or $\R \to \{ 0, \cdots, k\}$) is defined by setting $p|_{\Bc^p} := p$ for all $1 \leq p \leq k$, and $p|_{\Ic \cup \Gc} = 0$. Next, for an interval $I \subset S^1$ or $\R$, we define 
\[
p(I) = \max_{x \in I} p(x) \, .
\]
For $\omega \in [- \e, \e]$, we define $p_\omega(\cdot) = p(\cdot - \omega)$. 

\begin{rmk}\label{rmk:smallAtoms}
For an atom $C \in \Pc$ or $\Pc_\omega$, write $C^+$ for the union of $C$ with its two adjacent atoms. Observe that for any interval $I$, we have that each atom $J \in \Pc_\omega(I)$ 
is contained in $C^+$ for some $C \in \Pc_\omega(I)$. By this line of reasoning, for any $J \in \Pc_\omega(I)$ with $p = p(J) \in \{ 1,\cdots, k-1\}$, we have the estimate 
\[
|J| \leq 6 p^{-2} L^{- \frac{p + 2}{2} - \b} \, .
\]
Similarly, if $J \in \Pc_\omega(I)$, $J \cap \Bc^k_\omega \neq \emptyset$ (i.e. $p(J) = k$) then $|J| \leq 3 \max\{ 1, K_1^{-1}\} L^{- \frac{k}{2} - \b}$.

For a lower bound: if in the above setting we have that there are at least two distinct atoms in $\Pc_\omega(I)$, then any atom $J \in \Pc_\omega(I)$ with $p = p_\omega(J) > 0$ must contain an atom $C \in \Pc_\omega|_{\Bc^p}$. Thus
\[
|J| \geq (p + 1)^{-2} L^{- \frac{p + 3}{2} - \b} \, .
\]
\end{rmk}

\begin{rmk}\label{rmk:extendBoundPeriods}
Fix a sample $\uo \in \Omega$ and let $J$ be a connected interval contained in $C^+$ for some $C \in \Pc_{\omega_0}$. If $p := p_{\omega_0}(J) > 0$, then 
\[
\tilde f^i_\uo(J) \subset \Gc 
%\quad \text{ and } |f^i_\uo(J)| < L^{- \b / 2} 
\quad \text{ for all } 1 \leq i \leq p \, ,
\]
even though $J$ is not necessarily a subset of $\Bc^p_{\omega_0}$. This is because $\Pc|_{\Bc^{p-1}}$-atoms are small enough so that $J \subset B(- \frac{p + \b}{2})$ must hold, Lemma \ref{lem:boundPeriodCoherence} implies that $f^i_\uo(B(- \frac{p + \b}{2})) \subset \Gc$ for all $1 \leq i \leq p$ and all samples $\uo$. Note, in particular, that $\tilde f^i_\uo(J)$ meets at most one component of $\Gc$ for each $1 \leq i \leq p$, hence $\Pc_{\omega_{i }} ( \tilde f^i_\uo(J)) = \{ \tilde f^i_\uo(J) \}$.
%and that the $f^i_\uo$-image of each connected component of $B(- \frac{l + \b}{2})$ had length $< L^{- \b/2}$. 
\end{rmk}

%if $N \leq 3$ then $\Pc_\omega(I)$ is the trivial partition, and $I \subset C^+$ for some $C \in \Pc_\omega(I)$, while if $N > 3$ then $\Pc_\omega(I)$ is the partition of $I$ for which $J \approx C$ for each atom $J$ and for some $C \in \Pc_\omega(I)$. Here, for $C \in \Pc_\omega$, we write $C^+$ for the union of $C$ with its two adjacent atoms; for connected intervals $I$ we write $I \approx C$ if $C \subset I \subset C^+$.

%\medskip

%Observe that if $I$ satisfies $I \subset C^+$ for some $C \in \Pc_{\omega_1}$ and $p = p_{\omega_1}(I) > 0$, then $I$ undergoes a bound period of length $p$ in the sense of Lemma \ref{lem:boundPeriodCoherence} (see Remark \ref{rmk:extendBoundPeriods}).

%\smallskip
%
%For a bounded, connected interval $I \subset S^1$ (or $\subset \R$) which is not a singleton, we define the partition $ \Pc_\omega(I)$ of $I$ as follows. To start, form $ \Pc_\omega|_{I} = \{ \hat J = J \cap I \text{ for some } J \in  \Pc_\omega\}$, and write $J_1, J_1, \cdots, J_{N}$ for the non-singleton atoms of this partition in increasing order from left to right (note that $N = 1$ is possible).
%\begin{itemize}
%	\item If $N = 1,2$ or $3$, then set $\Pc_\omega(I) := \{ I \}$.
%	\item If $N \geq 4$, then set $\Pc_\omega(I) = \{ J_1 \cup J_2, J_3, J_4, \cdots, J_{N - 2}, J_{N-1} \cup J_N\}$.
%\end{itemize}
%If $N \leq 3$ then $\Pc_\omega(I)$ is the trivial partition, and $I \subset C^+$ for some $C \in \Pc_\omega(I)$, while if $N > 3$ then $\Pc_\omega(I)$ is the partition of $I$ for which $J \approx C$ for each atom $J$ and for some $C \in \Pc_\omega(I)$.

\subsubsection*{(B)  Time-$n$ itineraries for an interval $I \subset S^1$.}

Let $I \subset S^1$ be an interval (which we regard as a subset of $\R$) and fix a sample $\uo \in \Omega$. For each time $i \geq 1$, we
define a partition $\Qc_i = \Qc_i(I; (\omega_0, \cdots, \omega_i))$ of $I$, the atoms of which correspond to points in $I$ with the same itinerary for the map $\tilde f^{i + 1}_\uo$. 

The definition is inductive. To start, we define $\Qc_0 = \Pc_{\omega_0}( I)$. Assuming $\Qc_0, \Qc_1, \cdots, \Qc_i$ have been constructed, for each $C_i \in \Qc_i$ we define $\Qc_{i +1} \geq \Qc_i$ as follows\footnote{{ Here, for two partitions $\zeta, \xi$, we write $\zeta \leq \xi$ if each atom of $\zeta$
is a union of $\xi$-atoms.}}:
\[
\Qc_{i + 1}|_C = (\tilde f^{i + 1}_\uo)^{-1} \big(\Pc_{\omega_{i + 1}} (\tilde f^{i+1}_\uo(C_i)) \big)  \, .
\]

In what follows, we will only attempt to keep track of itineraries until a first ``near visit'' to the set $\Bc^k$. Precisely, we define
 a `terminating' stopping time $\tau = \tau[I] : I \times \Omega \to \Z_{\geq 0} \cup \{ \infty\}$ as follows:
\[
\tau (x, \uo) =  \min\{ i \geq 0 : f^i_\uo (C_i(x)) \cap \Bc^k_{\omega_{i }} \neq \emptyset \} \, .
\]
Here, $C_i(x)$ denotes the $\Qc_i$-atom containing $x$. Notice that $\tau$ is adapted to $(\Qc_i)_i$, i.e., $\{ \tau > i\}$ is a union
of $\Qc_i$-atoms for each $i \geq 0$. In particular, $\{ \tau > i\}$ depends only on $\omega_0, \cdots, \omega_i$.
%The time $\tau$ indicates when an itinerary gets too close to the `worst' bad set $\Bc^k$; after time $\tau$, we lose control over that itinerary.  

%We define here an increasing sequence of 
%partitions $\Qc_0 = \{ I \} \leq \Qc_1 \leq \Qc_2 \leq \cdots$ of $I$, $\Qc_i = \Qc_i(I; \omega_1, \cdots, \omega_i)$, 
%recording 
%keeping in mind the connotation that two points $x,y \in I$ belonging to the same $\Qc_i$ atom have the same itinerary
%up to time $i$. 

%\begin{defn}
%At a given time $i$, each of the $C \in \Qc_i$ is \emph{bound}
%if $f^{j}_\uo(C)$ is contained inside $\mathcal{G}$, where %$1\leq j\leq i$. 
% roughly speaking, the image $f^i_\uo(C)$ is undergoing a bound period in the sense of 
%Lemma \ref{lem:boundPeriodCoherence}. 
%{\color{green}I specify the definition of bound interval. }We say $C$ is \emph{free} if $C$ is not bound. 
%\end{defn}
%For a time $i \geq 0$ and an atom $C \in \Qc_i$, we have that $C$ is either (a) bound at time $i$, (b) free at time $i$, or (c) has $\tau|_C \leq i$.

%To start, we set $\Qc_0 := \{ I \}$. Assume that $\Qc_0, \cdots, \Qc_i$ have been constructed.
%For $C \in \Qc_i$ we define $\Qc_{i + 1}|_C$ as follows: define $I_i = \tilde f^i_\uo (C)$ and form the
%partition $\Pc_{\omega_{i + 1}}(I_i)$ of $I_i$. We then set
%\[
%\Qc_{i + 1}|_C := (\tilde f^i_\uo)^{-1} (\Pc_{\omega_{i + 1}}(I_i))
%\]
%Observe that the partition $\Qc_i$ depends only on the samples $\omega_1, \cdots, \omega_i$. Clearly, $\Qc_{i + 1} \geq \Qc_i$ holds.

\subsubsection*{(C) Bound and free periods of an itinerary}

Fix $n \geq 1$ and $C_n \in \Qc_n$ such that $\tau|_{C_n} \geq n$. For each $i < n$, let $C_i \in \Qc_i$ denote the atom containing $C_n$. For 
$1 \leq i \leq n$, we write $I_i = \tilde f^i_\uo(C_i)$.
%and $\check I_i = \tilde f^i_\uo(C_{i + 1}) \subset I_i$ (noting that $\check I_i$ depends on $\omega_{i + 1}$, while
%$I_i$ depends only on $\omega_1, \cdots, \omega_i$).

Define
\begin{gather}\label{defineTJ}\begin{gathered}
t_1 = \min \{ n \} \cup \{ i \geq 0 :  I_i  \cap \Bc_{\omega_{i }} \neq \emptyset \} \, ,  \quad \text{and} \\
t_j = \min \{ n \} \cup \{ i > t_{j-1} : I_i \cap \Bc_{\omega_{i }} \neq \emptyset \} \,  \quad \text{ for } j \geq 2 \, ,
\end{gathered}\end{gather}
and let $q \geq 0$ be the index for which $t_{q + 1} = n$. For $1 \leq j \leq q$, define
\begin{align}\label{definePJ}
p_j = p_{\omega_{t_j }}( I_{t_j}) \, .
\end{align}
At time $t_j, 1 \leq j \leq q$, the itinerary $C_n$ initiates a bound period of length $p_j$ (Remark \ref{rmk:extendBoundPeriods}); in particular,
 $t_j + p_j < t_{j +1}$ for all $1 \leq j < q$. We say that $C_n$ is \emph{bound at time $t$} if $t \in [t_j + 1, t_j + p_j]$ for some $1 \leq j < q$ and
 that $C_n$ is \emph{free at time $t$} if it is not bound at time $t$.

By Remark \ref{rmk:extendBoundPeriods} and the fact that $\tau|_{C_n} \geq n$, we have the following.
\begin{lem}\label{lem:fullBoundPer}
Let $1 \leq i \leq n$ and assume $C_n \in \Qc_n$ is such that $\tau|_{C_n} \geq n$.
\begin{itemize}
\item[(a)] If $C_n$ is free at time $i$, then
\[
|(f^i_\uo)'|_{C_n}| \geq L^{i (\frac12 - \b)} \, .
\]
\item[(b)] If $C_n$ is bound at time $i$, i.e., $i \in [t_j + 1, t_j + p_j]$ for some $1 \leq j \leq q$, then
\[
|(f^i_\uo)'|_{C_n}| \geq L^{t_j (\frac12 - \b) + (1 - \b)(i - (t_j + 1)) - \frac{p_j - 1}{2} - \b} \, .
\] 
In this case, $C_{t_j} = C_{t_j + 1} = \cdots = C_{t_j + p_j} = C_i$ and $C_n$ is free at time $t_j + p_j + 1$. Note that
$C_{t_j + p_j + 1} \subsetneq C_i$ is possible.
\end{itemize}
\end{lem}

\subsection{Distortion estimates}

%{\color{red} Fix a sample $\uo \in \Omega$ and let $I \subset S^1$ be a small interval. We describe here a procedure for partitioning $I$ according to itineraries with respect to $\tilde f^n_\uo$ for $n \geq 1$.
%}
%Let $\uo \in \Omega$ be fixed. Fix $x, x' \in S^1$ and a sample $\uo \in \Omega$. For each $m$ let $I_m$ denote the interval from $\tilde X_m$ to $\tilde X_m'$, where $\tilde X_m = \tilde f_\uo^m x, \tilde X_m' = \tilde f_\uo^m x'$.
%
%We say that $x, x'$ have the same \emph{random itinerary to time $n$} if the following holds: there are times $0 \leq t_1 < t_1 + p_1 + 1 \leq t_2 < \cdots < t_q + p_q + 1 \leq n$ such that (1) we have $I_{t} \subset C^+$ for some $C \in \Pc_{\omega_{t + 1}}$ for all $t < n$, $t \notin \cup_{j = 1}^l (t_j, t_j + p_j + 1]$; (2) for each $j = 1, \cdots, q$, we have $p_j = p_{\omega_{t_j + 1}}(I_{t_j})$; and (3) $I_t \cap \Bc_k = \emptyset$ for all $0 \leq t < n$. {\color{red} What if $I_n$ is bound??? Describe this better}

%Below, we assume that $I \subset S^1$ (or $\R$) is an interval, and that the construction of the $\Qc_i$ has been carried out as in Sect. 2.3.

Let $I \subset S^1$ be a connected interval, $\uo \in \Omega$ a sample. Assume that the partitions $(\Qc_i)_{i \geq 0}, \Qc_i = \Qc_i(I; (\omega_0, \cdots, \omega_i))$ and the stopping time $\tau = \tau[I]$ have
been constructed as in Section 4.1. Here we prove a time-$n$ distortion estimate for trajectories with the same time-$n$ itineraries, i.e., belonging to the same $\Qc_n$-atom.

Our approach to distortion estimates is inspired from the treatment in \cite{WY},
which in turn is a version of estimates first appearing in \cite{BC1,BC2}.

%A difference here is that we only control distortion along $I$ hold only up to a `terminating' stopping time $\tau : I \to \N \cup \{ \infty\}$, defined as follows:
%\[
%\tau = 1 + \min\{ i \geq 0 : C_i(x) \cap \Bc^k_{\omega_{i + 1}} \neq \emptyset \} \, .
%\]
%Here, $C_i(x)$ denotes the $\Qc_i$-atom containing $x$. Notice that $\tau$ is adapted to the $\Qc_i$, i.e., $\{ \tau > i\}$ is a union
%of $\Qc_i$-atoms. The time $\tau$ indicates when an itinerary gets too close to the `worst' bad set $\Bc^k$; after time $\tau$,
%we lose control over that itinerary.  
%
%{\color{green}I add the definition of bound and free here.}
%\begin{defn}Fix a sequence $\underline{\omega}=\{\omega_i\}$. An interval $I$ is called bound if $\tilde{f}_{\omega_i}(I)\subset\mathcal{G}$. Moreover, if $I\cap \mathcal{B}\neq\emptyset,$ we call $I$ is bound with bound period $p_{\underline{\omega}_n}(I)$. An interval $I$ is called free at time $n$ if $f^n(I)\cap \mathcal{B}=\emptyset.$. For simplicity, when $I\cap \mathcal{B}=\emptyset,$ we call $I$ free. An interval $I$ is called free with free period $n$ if $f^i(I)\cap \mathcal{B}=\emptyset, 0\leq i\leq n-1$ while $f^n(I)\cap \mathcal{B}\neq \emptyset. $ 
%
%\end{defn}

%{\color{green} Since $I$ may not be a real subset of $\mathcal{B}$, I am worried about the definition of $p_{\underline{\omega}_n}(I)$. This should involve the way you define the boundary in the partition $\mathcal{P}$. }

\begin{prop}\label{prop:distortion}
For all $L$ sufficiently large, the following holds. 
Let $n \geq 1$. Assume $C_n \in \Qc_n$ is free at time $n$ and  $\tau|_{C_n} \geq n$.
Let $x, x' \in C_n$. Then,
\[
\frac{(\tilde f^n_\uo)'(x)}{(\tilde f^n_\uo)'(x')} \leq e^{K_2 L^{-\frac12} + 4 \| \psi'' \|_{C_0} L^{2 \b} |\tilde f^n_\uo x - \tilde f^n_\uo x'|} \, .
\]
\end{prop}

We start with a preliminary Lemma.

\begin{lem}\label{lem:shortDist}
Let $L$ be sufficiently large, and let $\eta \in [-\frac34,0]$. Let $y, y' \in S^1, i \geq 1$, and define $J$ to be the interval between $y, y'$.  If $f^j_\uo(J) \subset B(\eta)^c$ for all $0 \leq j < i$, then
\[
\bigg| \log \frac{ (f^i_\uo)'(y)}{(f^i_\uo)'(y')} \bigg| \leq 2 \| \psi''\|_{C^0} L^{-1 - 2 \eta} |\tilde f^i_\uo (y) - \tilde f^i_\uo (y')| \, .
\]
\end{lem}
\begin{proof}
Define $y_j = \tilde f^j_\uo y, y_j' = \tilde f^j_\uo y'$. We estimate
\[
(*) := \bigg| \log \frac{ (f^i_\uo)'(y)}{(f^i_\uo)'(y')} \bigg| \leq \sum_{j = 0}^{i-1} \bigg| \log \frac{ (f_{\omega_{j + 1}})'(y_j)}{(f_{\omega_{j + 1}})'(y_j')} \bigg| 
\leq \sum_{j = 0}^{i-1} \frac{L \| \psi'' \|_{C_0}}{ L^{1 + \eta}} |y_j - y_j'| = \| \psi''\|_{C^0} L^{- \eta} \sum_{j = 0}^{i-1} |y_j - y_j'| \, .
\]
We bound $|y_j - y_j'| \leq L^{-(1 + \eta)(i - j)} |y_i - y_i'|$, hence
\[
(*) \leq \| \psi''\|_{C^0} L^{- \eta} \bigg(  \sum_{j = 0}^{i-1} L^{- (1 + \eta)(i - j)} \bigg) |y_i - y_i'| \leq 2 \| \psi''\|_{C^0} L^{-1 - 2 \eta} |y_i - y_i'| \, .
\]
In view of \eqref{eq:lowerBoundDer}, observe that the above estimates can be written in the following alternative form: writing $J_j$ for the interval between $y_j, y_j'$, we have that
\[
\sum_{j = 0}^{i-1} \frac{|J_j|}{d(J_j, C_\psi' - \omega_j)} \leq 2 \| \psi'' \|_{C^0} L^{-1 - 2 \eta} |J_i| \, .
\]
%in view of \eqref{eq:lowerBoundDer}.
\end{proof}

\begin{proof}[Proof of Proposition \ref{prop:distortion}]
Below, we write $C$ to refer to a generic positive constant; 
the value of $C$ may change from line to line, but always depends only on the function $\psi$.

With $n \geq 1$ and $C_n \in \Qc_n$ fixed and free at time $n$, we adopt the notation of Section 4.1 (C). 
Write $x_i = \tilde f^i_\uo (x), x_i' = \tilde f^i_\uo(x')$. By hypothesis, $x, x'$ belong to the same $\Qc_i$ 
atom $C_i$ for all $0 \leq i \leq n$.  

We decompose
\[
\bigg| \log \frac{(\tilde f^n_\uo)'(x)}{(\tilde f^n_\uo)'(x')}\bigg| \leq \sum_{i = 0}^{n-1} \bigg| \log \frac{\tilde f_{\omega_{i }'}(x_i)}{\tilde f_{\omega_{i }'}(x_i')} \bigg|
\]
Using \eqref{eq:lowerBoundDer}, each summand is bounded by
\[
\bigg| \log \frac{\tilde f_{\omega_{i}}'(x_i)}{\tilde f_{\omega_{i }}'(x_i')} \bigg| \leq C \frac{|J_i|}{d(J_i, C_\psi' - \omega_i)} \, ,
\]
where $J_i$ is the interval from $x_i + \omega_i$ to $x_i' + \omega_i$.

With $t_j, p_j$ as in \eqref{defineTJ},\eqref{definePJ}, we decompose the time interval from 
$0$ to $n$ into the succession of free and bound periods experienced by the atom $C_n \in \Qc_n$ containing $x, x'$: 
\[
0 \leq t_1 < t_1 + p_1 < t_2 < t_2 + p_2 < \cdots < t_q < t_q + p_q < t_{q + 1} := n \, .
\]
We assume going forward that $q \geq 1$, i.e., $C_n$ experiences at least one bound
period. If $q = 0$, then Proposition \ref{prop:distortion} follows easily from Lemma \ref{lem:shortDist}
applied to $\eta = - \frac12 - \b$; details are left to the reader.

We now decompose $\sum_{i = 0}^{n-1}$ as follows:
\[
\sum_{i = 0}^{n-1} \frac{|J_i|}{d(J_i, C_\psi' - \omega_i)} = \sum_{i = 0}^{t_1 -1} + \sum_{j =1 }^q \bigg(  \sum_{i = t_j}^{t_j + p_j} + \sum_{i = t_j + p_j + 1}^{t_{j + 1} - 1} \bigg) = : D_0' + \sum_{j = 1}^q (D_j + D_j')
\]
%where $D_0' = \sum_{i = 0}^{t_1 - 1}$, $D_j = \sum_{i = t_j}^{t_j + p_j}$ and 
%$D_j' = \sum_{i = t_j + p_j + 1}^{t_{j + 1} - 1}$ (if $t_j + p_j + 1 = t_{j + 1}$, then $D_j'$ is defined to be $0$).
Above, a summand of the form $\sum_{m}^{m-1}, m \in \N$ is regarded as 
empty and the corresponding summation is defined to be $0$ (as may happen
for some of the $D_j'$ terms).
The $D_j, D_j'$ are estimated separately below.
%\[
%D_i = \log \frac{(\tilde f^{p_j + 1}_{\theta^{t_j} \uo})'(x_{t_j})}{(\tilde f^{p_j + 1}_{\theta^{t_j} \uo})'(x_{t_j}')}
%\]
%\begin{align*}
%\log \frac{(\tilde f^n_\uo)'(x)}{(\tilde f^n_\uo)'(x')} & = \log \frac{(\tilde f^{t_1}_\uo)'(x)}{(\tilde f^{t_1}_\uo)'(x')} 
%+  \sum_{j = 1}^{q-1} \bigg( \log \frac{(\tilde f^{p_j + 1}_{\theta^{t_j} \uo})'(x_{t_j})}{(\tilde f^{p_j + 1}_{\theta^{t_j} \uo})'(x_{t_j}')} + \log \frac{(\tilde f^{t_{j + 1} - (t_j + p_j + 1)}_{\theta^{t_j + p_j + 1} \uo })'(x_{t_j + p_j + 1}) }{(\tilde f^{t_{j + 1} - (t_j + p_j + 1)}_{\theta^{t_j + p_j + 1} \uo })' (x_{t_j + p_j + 1}') }\bigg) \\
%& + + \frac{(\tilde f^{n - (t_l + p_l + 1)}_{\theta^{t_l + p_l + 1} \uo} )'(x_{t_l + p_l + 1})}{(\tilde f^{n - (t_l + p_l + 1)}_{\theta^{t_l + p_l + 1} \uo} )'(x_{t_l + p_l + 1}')} \\
%& =: D_0' + \sum_{j = 1}^q (D_j + D_j') + D_q + D_{q}'
%\end{align*}
%
%\[
%\log \frac{(\tilde f^n_\uo)'(x)}{(\tilde f^n_\uo)'(x')} \leq \sum_{j = 0}^{n-1} \log \frac{ (\tilde f_{\omega_{j + 1}})'(x_j) } {(\tilde f_{\omega_{j + 1}})'(x_j')} 
%\leq \frac{1}{K_1} \sum_{j = 0}^{n-1} \frac{|x_j - x_j'|}{\dist(I_j, C_\psi')} \leq K_1^{-1} \bigg( D_0' + \sum_{i = 1}^q \big( D_i + D_i' \big) \bigg) 
%\]
%where
%\[
%D_0' = \sum_0^{t_1 - 1} \frac{|x_j - x_j'|}{\dist(I_j, C_\psi')} \, , \quad D_i = \sum_{j = t_i}^{t_i + p_i} \frac{|x_j - x_j'|}{\dist(I_j, C_\psi')} \, , \quad D_i' = \sum_{j = t_i + p_i + 1}^{t_{i + 1} - 1} \frac{|x_j - x_j'|}{\dist(I_j, C_\psi')} \, ,
%\]
%except for $D_q'$, for which the summation goes from $t_q + p_q + 1$ to $n-1$.

\medskip

Before proceeding, observe that $|J_{t_j + p_j + 1}| \geq L^{(p_j + 1)(\frac12 - \beta)} |J_{t_j}|$ and $|J_{t + 1}| \geq L^{\frac12 - \b} |J_t|$ for all $t$ such that $C_t, C_{t + 1}$ are free. %all $t \in [0,t_1) \cup \big(\cup_{j = 1}^{q-1} [t_j + p_j + 1, t_{j + 1})) \big) \cup [t_q, n)$. 
In particular, 
\begin{align}\label{eq:boundJ}
|J_{t_{j + 1}}| \geq L^{(t_{j + 1} - t_j) (\frac12 - \beta)} |J_{t_i}|
\end{align}
for all $1 \leq i \leq q$. 

\bigskip

\noindent {\it Bounding $\sum_{j = 1}^q D_j$: } Let $1 \leq j \leq q$. 
\begin{cla}
\[
\sum_{i = t_j + 1}^{t_j + p_j} \frac{|J_i|}{\dist(J_i, C_\psi' - \omega_i)} \leq C L^{2 \b} \frac{|J_{t_j}|}{d(J_{t_j}, C_\psi' - \omega_{t_j})}
\]
\end{cla}
\noindent Assuming the Claim, we now bound $\sum_{j = 1}^q D_j$. For $1 \leq p < k$, let $\mathcal K_p = \{ 1 \leq j \leq q : p_j = p\}$. Let $j^*_p = \max \mathcal K_p$, and observe that $|J_{t_j}| \leq |J_{t_{j^*_p}}| \cdot L^{- (t_{j^*_p} - t_j) (\frac12 - \beta)}$ for all $j \in \mathcal K_p$ by \eqref{eq:boundJ}. Thus
\[
 \sum_{j \in \mathcal K_p} D_j \leq C L^{2 \b} \sum_{j \in \mathcal K_p} \frac{|J_{t_j}|}{\dist(J_{t_j}, C_\psi' - \omega_{t_j})} 
%\leq \sum_{i \in \mathcal K_p} \frac{|J_{t_{i^*_p}}| L^{- (i^*_p - i) (\frac12 - \beta)}}{\frac12 \d L^{- \frac{p + 1}{2}}} 
\leq \frac{C L^{2 \b}}{1 - L^{- (\frac12 - \beta)}} \cdot \frac{|J_{t_{j_p^*}}|}{ \frac12 K_1^{-1} L^{- \frac{p + 1}{2} - \b}} \leq C L^{2 \b} \frac{|J_{t_{j_p^*}}|}{ L^{- \frac{p + 1}{2} - \b}} \, .
\]
Here we are using that $\dist(J_{t_j}, C_\psi' - \omega_{t_j}) \geq \frac12 K_1^{-1} L^{- \frac{p + 1}{2} - \b}$ for all $j \in \mathcal K_p$. 
By Remark \ref{rmk:smallAtoms}, we have $|J_{t_{j^*_p}}| \leq 6 p^{-2} L^{- \frac{p + 2}{2} - \b}$. So,
\[
\sum_{j \in \mathcal K_p} D_j \leq C L^{2 \b}
%2 (1 + 8 \| \psi '' \|_{C^0} c_1^{-1})
 \frac{ p^{-2} L^{- \frac{p + 2}{2} - \b}}{ L^{- \frac{p + 1}{2} - \b}}  \leq C p^{-2} L^{- \frac12 + 2 \b}
\]
hence
\[
\sum_{j = 1}^q D_j = \sum_{p = 1}^{k-1} \sum_{j \in \mathcal K_p} D_j \leq \sum_{p = 1}^{k-1} C p^{-2} L^{- \frac12 + 2 \b} \leq C L^{- \frac12 + 2 \b} \, .
\]

\begin{proof}[Proof of Claim]
Assume $I_{t_j}$ meets the component of $\Bc_{\omega_{t_j }}$ near $\hat x_{t_j} \in C_\psi' - \omega_{t_j}$; write $\hat x_i = \tilde f^{i - t_j}_{\theta^{t_j} \uo} (\hat x_{t_j})$ for $i > t_j$. Assume, without loss, that 
\begin{align}\label{endpointAnsatz}
|x_{t_j}' - \hat x_{t_j}| \leq |x_{t_j} - \hat x_{t_j}| \, ;
\end{align}
in the alternative case, exchange the roles of $x_i, x_i'$ in what follows. 

For $t_j < i \leq t_j + p_j$, we have
\[
\frac{|J_i|}{\dist(J_i, C_\psi' - \omega_i)} = \frac{|x_i - x_i'|}{|x_i - \hat x_i|} \cdot \frac{|x_i - \hat x_i|}{\dist(J_i, C_\psi' - \omega_i)}
\]
By Lemmas \ref{lem:boundPeriodCoherence} and \ref{lem:shortDist}, we have that the first right-hand factor is
\begin{align}\label{eq:firstFactor}
\leq 2 \frac{|x_{t_j+ 1} - x_{t_j+ 1}'|}{|x_{t_j+ 1} - \hat x_{t_j+ 1}|}
\end{align}
The numerator of \eqref{eq:firstFactor} coincides with $|f_{\omega_{t_j}}'(\zeta)| \cdot |x_{t_j} - x_{t_j}'|$ for some $\zeta \in J_{t_j}$. Moreover, $|f_{\omega_{t_j}}'(\zeta)| = |f_{\omega_{t_j}}''(\zeta')| \cdot |\zeta - \hat x_{t_j}| \leq L \| \psi'' \|_{C^0} |\zeta - \hat x_{t_j}|$ for some $\zeta'$ between $\zeta$ and $\hat x_{t_j}$. By \eqref{endpointAnsatz} 
%$\zeta \in J_{t_j}$ and $x_{t_j}$ is farther from $\hat x_{t_j}$ than the other endpoint $x_{t_j}'$, 
we have $|\zeta - \hat x_{t_j}| \leq |x_{t_j} - \hat x_{t_j}|$, and so conclude that the numerator of \eqref{eq:firstFactor} is $\leq L \| \psi'' \|_{C^0}\cdot  |x_{t_j} - \hat x_{t_j}| \cdot |J_{t_j}|$. 

For the denominator of \eqref{eq:firstFactor}, we have
$
|x_{t_j+ 1} - \hat x_{t_j+ 1}| = \frac12 |f_{\omega_{t_j}}''(\zeta'')| |x_{t_j} - \hat x|^2
$
for some $\zeta''$ between $x_{t_j}$ and $\hat x$. For $L$ sufficiently large and all $\e$ satisfying \eqref{boundEpsilon}, we have that $\min_{z \in \Nc_\e (\Bc)} |\psi''(z)| \geq \frac12 \min \{ | \psi''(\hat z) |: \hat z \in C_\psi'\} =: c_1$ from (H1), (H2). We have therefore that the denominator of \eqref{eq:firstFactor} is $\geq \frac12 c_1 L |x_{t_j} - \hat x_{t_j}|^2$.

Collecting,
\[
\frac{|J_i|}{\dist(J_i, C_\psi' - \omega_i)} \leq C  \frac{|J_{t_j}|}{\dist(J_{t_j}, C_\psi' - \omega_{t_j})} \cdot \frac{|x_i - \hat x_i|}{\dist(J_i, C_\psi' - \omega_i)} \, ,
\]
since $|x_{t_j} - \hat x_{t_j}|^{-1} \leq d(J_{t_j}, C_\psi' - \omega_{t_j})^{-1}$ by assumption,
and so
\[
\sum_{i = t_j + 1}^{t_j + p_j}  \frac{|J_i|}{d(J_i, C_\psi' - \omega_i)} \leq C \frac{|J_{t_j}|}{\dist(J_{t_j}, C_\psi' - \omega_{t_j})} \bigg( \sum_{i = t_j + 1}^{t_j + p_j} \frac{|x_i - \hat x_i|}{\dist(J_i, C_\psi' - \omega_i)} \bigg)\, .
\]
By Lemma \ref{lem:shortDist} applied to $\eta = - \b$, the parenthetical sum is bounded $\leq C L^{-1 + 2 \b} |x_{t_j + p_j + 1} - \hat x_{t_j + p_j + 1}|$. Since $|x_{t_j + p_j} - \hat x_{t_j + p_j}| \leq L^{- \b/2} \ll 1$ (see the proof of Lemma \ref{lem:boundPeriodCoherence}), we bound $|x_{t_j + p_j + 1} - \hat x_{t_j + p_j + 1}| \leq C L$, hence the parenthetical sum is $\leq C L^{2 \b}$. This completes the proof.
\end{proof}

%and so
%\[
%D_j \leq \frac{|J_{t_j}|}{\dist(J_{t_j} , C_\psi')} \bigg( 1 + C \sum_{i = t_j + 1}^{t_j + p_j} \frac{|x_i - \hat x_i|}{\dist(J_i, C_\psi')}\bigg) \leq C  \frac{|J_{t_j}|}{\dist(J_{t_j} , C_\psi')} \, ,
%\]
%where above we estimate the $\sum_{i = t_j + 1}^{t_j + p_j}$ terms as in Lemma \ref{lem:shortDist}. 
%

\medskip

\noindent {\it Bounding $\sum_{j = 0}^q D_j'$: } For each $1 \leq j < q$, we have from Lemma \ref{lem:shortDist} applied to $\eta = - \frac12 - \b$ that
\[
D_j' \leq L^{1 - 2 (-\frac12 - \b)} |J_{t_{j + 1}}| = C L^{2 \b} |J_{t_{j + 1}}| \, .
\]

%we have
%\begin{align*}
%D_j' & = \sum_{i = t_j + p_j + 1}^{t_{j + 1}-1} \log \frac{\tilde f_{\omega_{i + 1}}'(x_i)}{\tilde f_{\omega_{i + 1}}'(x_i')} 
%= \sum_{i= t_j + p_j + 1}^{t_{j  +1} - 1} \frac{L \| \psi'' \|_{C^0}}{L^{\frac12}} |x_i - x_i'| \\
%& \leq \sum_{i = t_j + p_j + 1}^{t_{j  +1} - 1} L^{\frac12} \| \psi'' \|_{C^0} \cdot L^{- \frac12 (t_{j + 1} - i)}  |x_{t_{j + 1}} - x_{t_{j + 1}}'|  \leq \frac{\| \psi''\|_{C^0} }{1 - L^{- \frac12}} |x_{t_{j + 1}} - x_{t_{j + 1}}'| \\
%& \leq 2 \| \psi'' \|_{C^0} |J_{t_{j + 1}}| \, .
%%& \frac{|I_j|}{\dist(I_j, C_\psi')} \leq L^{\frac12 - \beta} \sum_{j = t_i + p_i + 1}^{t_{i + 1} - 1} |I_j| \leq L^{\frac12 + \beta} \sum_{j = t_i + p_i + 1}^{t_{i + 1} - 1} L^{- (j - t_{i + 1})(\frac12 - \beta)} |I_{t_{i + 1}}| \\
%%& \leq \frac{L^{2 \beta}}{1 - L^{- (\frac12 - \beta)} } |I_{t_{i + 1}}| \leq 2 L^{2 \beta} |I_{t_{i + 1}}| 
%% \frac{1}{1 - L^{- (\frac12 - \beta)}} 
%\end{align*}

Similarly, we estimate $D_0' \leq C L^{2 \b} |J_{t_1}|$. Since $|J_{t_j}| \leq L^{- (n - t_j)(\frac12 - \beta)} |J_n|$ for all $1 \leq j \leq q$ by \eqref{eq:boundJ}, we conclude 
%$\sum_{i = 0}^{q} D_j' \leq \frac{2 \| \psi'' \|_{C^0}}{1 - L^{- \frac12 + \beta}} |J_n| \leq 4 \| \psi'' \|_{C^0} |J_n|$.
$\sum_{i = 0}^{q} D_j' \leq C L^{2 \b} |J_n|$. The proof of Proposition \ref{prop:distortion} is now complete.
\end{proof}

%\begin{align*}
%\log \frac{(\tilde f^n_\uo)'(x)}{(\tilde f^n_\uo)'(x')} & = \log \frac{(\tilde f^{t_1}_\uo)'(x)}{(\tilde f^{t_1}_\uo)'(x')} 
%+  \sum_{j = 1}^l \bigg( \log \frac{(\tilde f^{p_j + 1}_{\theta^{t_j} \uo})'(x_{t_j})}{(\tilde f^{p_j + 1}_{\theta^{t_j} \uo})'(x_{t_j}')} + \log \frac{(\tilde f^{t_{j + 1} - (t_j + p_j + 1)}_{\theta^{t_j + p_j + 1} \uo })'(x_{t_j + p_j + 1}) }{(\tilde f^{t_{j + 1} - (t_j + p_j + 1)}_{\theta^{t_j + p_j + 1} \uo })' (x_{t_j + p_j + 1}') }\bigg) \\
%& + \frac{(\tilde f^{n - (t_l + p_l + 1)}_{\theta^{t_l + p_l + 1} \uo} )'(x_{t_l + p_l + 1})}{(\tilde f^{n - (t_l + p_l + 1)}_{\theta^{t_l + p_l + 1} \uo} )'(x_{t_l + p_l + 1}')}
%\end{align*}

%\begin{rmk}
%{\color{red} Nothing done yet about problem of landing in the set $\Bc_k$; re-smearing scheme introduced in the next section.}
%\end{rmk}

%\end{document}

\section{Selective averaging process}

We aim to get more refined control on the conditional laws $P^n(X_0, \cdot | \{ \omega_i, i \neq 0\}), n \geq 0$.
Towards this end, the itinerary subdivision procedure in Section 4 applied to $I = X_0 + [- \e, \e]$ 
can be used to control the density of $P^n(X_0, \cdot | \{ \omega_i, i \neq 0\}, X_0 + \omega_0 \in C_n)$ for some $C_n \in \Qc_n$, i.e., conditioning on $X_0 + \omega_0$ belonging to a single subdivision $C_n$ of $\Qc_n$. 
This is only valid, however, up until the first `near visit' to $\Bc^k$, the closest
neighborhood to the critical set $\{ f' = 0 \}$. Afterwards, the material in Section 4 is no longer
valid and we lose control over distortion, hence over the conditioned law $P^n(X_0, \cdot | \{ \omega_i, i \neq 0\})$.

A rough idea of how to proceed is as follows: visits to $\Bc^k$ `spoil' the random parameter $\omega_0$, and so
if $X_m$ comes too close to $\Bc^k$ for some $m \geq 0$, we will `freeze' $\omega_0$ 
(essentially, treat as deterministic) and `smear' (average) in the perturbation $\omega_{m + 1}$, i.e., 
for $n \geq m$, work with the conditional law $P^n(X_0, \cdot | \{ \omega_i, i \neq m + 1\})$.

\bigskip

Let us make all this more precise. Fix $X_0 \in S^1$ and define the Markov chain $(\tilde X_n)$ on $\R$ by $\tilde X_n = \tilde f^n_\uo(X_0) = \tilde f_{\omega_{n-1}}(\tilde X_{n-1})$. We will obtain in this section an increasing filtration $(\Hc_n)_{n \geq 0}$, $\Hc_n \subset \Fc_{n} := \sigma(\omega_0, \omega_1, \cdots, \omega_n)$ (depending also on $X_0$), designed so that the conditional measures
\[
\nu_n(\cdot) := \P(\tilde X_n \in \cdot | \Hc_n)
\]
have the following desirable properties:
\begin{itemize}
\item[(i)] the measures $\nu_n$ are absolutely continuous; 
\item[(ii)] $\rho_n := \frac{d \nu_n}{d \Leb}$ is more-or-less constant on the interval of support $I_n := \operatorname{supp} \nu_n$; and 
\item[(iii)] the intervals $I_n = \operatorname{supp} (\nu_n)$ are, for large $n$, rather long with high probability.
\end{itemize}

In this section, we focus on the construction of $\Hc_n, I_n, \nu_n$ as above; 
property (ii) will fall out as a natural consequence of our construction and the 
distortion estimate in Proposition \ref{prop:distortion}.

\bigskip

The plan is as follows: first, in Section 5.1 we will describe an algorithm constructing the supporting intervals $I_n$ as above, in a way completely parallel to the itinerary construction given in Section 4.1. From this construction, it will be clear when `smearing' in a new $\omega_i$ is necessary: this decision is made according to a sequence $\tau_1 < \tau_2 < \cdots$ of stopping times roughly related to the first arrival to the neighborhood $\Bc^k$ (closely related to the stopping time $\tau$ as in Section 4.1). In Section 5.2 we will construct the filtration $(\Hc_n)$ and then describe the resulting conditional measures $\nu_n$ in Section 5.3.

\bigskip

\emph{ In addition to the preparations in 
Section \ref{subsubsec:basicSetup}, we assume the parameter $\e$ satisfies 
\eqref{boundEpsilon}, so that Lemma \ref{lem:boundPeriodCoherence} holds. No lower bound on $\e$ is assumed.
}

\subsection{The supporting intervals $I_n$}

We define here an interval\footnote{For our purposes, an \emph{interval} is a bounded, connected subset of $\R$, with either open or closed endpoints. Since we care only about $\P$-typical trajectories, we need not specify what to do with endpoints.}-valued stochastic process $(I_n)_{n \geq 1}$ for which $I_n \subset \R$ is $\Fc_{n}$-measurable for all $n$. 
%{\color{green} \sout{At each stage $n$ in the construction, if $I_n$ is not a singleton, then it is declared \emph{bound} if (as in Section 2.3) it is shadowing a postcritical orbit in the sense of Lemma \ref{lem:boundPeriodCoherence}, and \emph{free} otherwise.}}

\smallskip

\renewcommand{\tX}{{\tilde X}}

%{\color{green}I change the procedure here a little bit to make sure the coherence of the notations. Please check whether it is what you want. }

Embed $X_0 =: \tilde X_0 \in \R$ via the identification $S^1 \cong [0,1)$. Throughout, the dependence of the $I_n$ on the sample $\uo = (\omega_i)_{i \geq 0} \in \Omega$ is implicit (keeping in mind that $I_n$ depends on $\omega_i, 0 \leq i \leq n$). 

\bigskip

\noindent {\it Base cases:} We set $I_0 = \tX_0 + [- \e, \e]$. To determine $I_1$, there are two cases:
\begin{itemize}
	\item If $I_0 \cap \Bc^k_{\omega_0} \neq \emptyset$, then define $I_1 = \tilde X_1 + [- \e, \e]$.
	\item Otherwise, form $\Pc_{\omega_1}( \tilde f(I_0) )$ and let $I_1$ be the atom containing $\tilde X_1$.
\end{itemize}
Note that since $\e > 0$ is assumed to satisfy \eqref{boundEpsilon}, we have automatically
that $\Pc(I_0)$ consists of a single atom.

\bigskip

\noindent {\it Inductive step: } Assume the intervals $I_0, I_1, \cdots, I_n$ have been constructed, with $n \geq 1$.

\begin{itemize}
	\item[(a)] If $I_n \cap \Bc^k_{\omega_n} = \emptyset, I_{n-1} \cap \Bc^k_{\omega_{n-1}} = \emptyset$, then form $\Pc_{\omega_{n + 1}}( \tilde f_{\omega_n}(I_n))$ and define $I_{n + 1}$ to be the atom containing $\tilde X_{n + 1}$. 
	\item[(b)] If $I_n \cap \Bc^k_{\omega_n} \neq \emptyset$, then define $I_{n + 1} = \tilde X_{n + 1} + [- \e, \e]$.
	Form $\Pc_{\omega_{n + 2}} \big( \tilde f(I_{n + 1})\big) $ and let $I_{n + 2}$ be the atom containing $\tilde X_{n + 2}$
%	\item[(c)] If $I_{n-1} \cap \Bc^k_{\omega_{n-1}} \neq \emptyset$, then form $\Pc_{\omega_{n + 1}}(\tilde f(I_n))$ and let $I_{n + 1}$ be the atom containing $\tilde X_{n + 1}$. 
%	\item If $I_n \cap \Bc^k \neq \emptyset$: form the partition $\Pc(\tilde X_{n + 1} + [- \e, \e])$ and 
%	let $I_{n + 1}$ be the atom containing $\tilde X_{n + 1} + \omega_{n + 2}$. 
%	\item Otherwise, define $\hat I_{n + 1} = \tilde f_{\omega_{n + 1}}(I_n)$ 
%		and form the partition $\Pc (\hat I_{n + 1})$. We define $I_{n + 1}$ 
%		to be the atom of this partition containing $\tilde X_{n + 1}$. 
\end{itemize}
From Lemma \ref{lem:boundPeriodCoherence} and Remark \ref{rmk:extendBoundPeriods}, it is simple to check that cases (a) -- (b) are exhaustive and mutually exclusive. Note in case (b) that $I_{n +1} \subset \Gc_{\omega_{n + 1}}$ holds
(Lemma \ref{lem:boundPeriodCoherence} and \eqref{boundEpsilon}).

\begin{defn}
We define a sequence of $(\Fc_n)$-adapted stopping times $0 =: \tau_0 < \tau_1 < \tau_2 < \cdots$ as follows: for $i > 0$, set
\begin{align*}
\tau_i  =   \min \{ m > \tau_{i - 1} : I_{m} \cap \Bc^k_{\omega_{m}} \neq \emptyset \}  \, .
\end{align*}
\end{defn}
Observe that case (b) above is observed iff $n = \tau_i$ for some $i$.%, while case (c) is observed iff $n = \tau_i + 1$ for some $i$.

%\begin{itemize}
%	\item If $I_n = \{ \tilde X_n\}$ is a singleton: form $\mathcal P( \tilde X_n + [- \e, \e])$ and let $\check I_n$ denote the atom containing $\tilde X_n$.  Define $I_{n + 1} = \tilde f(\check I_n)$. 
%	\item If $I_n$ is an interval: form $\mathcal{P}_{\omega_{n + 1}}(I_n)$ and let $\check I_n$ denote the atom containing $\tilde X_n$.  
%	
%		\begin{itemize}
%			\item If $\check I_n \cap \Bc^k_{\omega_{n + 1}}  \neq \emptyset$, then set $I_{n+1} = \{ \tX_{n + 1}  \}$. 
%			\item Otherwise, define $I_{n + 1} = \tilde f_{\omega_{n + 1}}(\check I_n)$. 
%		\end{itemize}
%\end{itemize}

\medskip

As formulated below, between `near visits' to $\Bc^k$ (i.e., the times $\tau_1, \tau_2, \cdots$), 
the procedure defining the $(I_n)$ process is completely parallel to the itinerary construction in Section 4.1.
The proof is straightforward and left to the reader.

%\begin{align*}
%\tau_i & = \min \big\{ m > \tau_{i - 1} : I_m = \{ \tilde X_m\} \big\} \\
%& =  1 + { \min \{ m > \tau_{i - 1} : \check I_{m} \cap \Bc^k_{\omega_{m+1}} \neq \emptyset \} } \, .
%\end{align*}

\begin{lem}\label{lem:correspond}
Fix $i \geq 0$ and $0 \leq m < n$. 
\begin{itemize}
\item[(a)] On the event $S_{i, m, n} = \{ \tau_i = m , \tau_{i + 1} \geq n\}$, we have that the random interval $I_n$ is given as
\[
I_n = \tilde f^{n-m-2}_{\theta^{m + 2} \uo} \circ \tilde f(\hat C) \, ,
\]
where $\hat C$ is the atom of $\Qc_{n-m-1} (I_{m + 1}; (0, \omega_{m+2}, \cdots, \omega_n))$ containing $\tilde X_{m+1} + \omega_{m+1}$ (recall $I_{m + 1} = \tilde X_{m+1} + [- \e, \e]$).
\item[(b)] On the event $\{ \tau_i = m\}$, we have $I_{m + 1} = \tilde X_{m + 1} + [- \e, \e]$ and
\[
 \tau_{i + 1} = m + \tau[I_{m + 1}](\tilde X_{m + 1} + \omega_{m + 1}, \hat \uo) \, ,
\]
where $\tau[I_{m + 1}]$ is the stopping time 
as defined in Section 4.1 with
 $\hat \uo = (0, \omega_{m + 2}, \omega_{m + 3}, \cdots)$.
\end{itemize}
%Above, $0 \theta^{m+1} \uo = (0, \omega_{m + 2}, \omega_{m + 3}, \cdots)$.
\end{lem}

\subsection{Filtration $(\Hc_n)$}

We now construct $\Hc_n = \sigma(\Ac_n)$, where the measurable partition $\Ac_n$ on $\Omega$ is defined below. Each $\mathcal A_n$ will consist of $\Fc_n$-measurable atoms, and so will be treated here as a partition on the first $n+1$ coordinates  $(\omega_0, \cdots, \omega_n) \in [- \e, \e]^{n+1}$.

\medskip

To start, we set $\Ac_0 = \{ [- \e, \e]\}$ to be the trivial partition, and hereafter assume $n \geq 1$.  

\smallskip

%Define a sequence of $(\Fc_n)$-adapted stopping times $0 =: \tau_0 < \tau_1 < \tau_2 < \cdots$ as follows: for $i > 0$, we set
%\begin{align*}
%\tau_i & = \min \big\{ m > \tau_{i - 1} : I_m = \{ \tilde X_m\} \big\} \\
%& =  1 + { \min \{ m > \tau_{i - 1} : \check I_{m} \cap \Bc^k_{\omega_{m+1}} \neq \emptyset \} } \, .
%\end{align*}
%These have the connotation of `smear times'-- e.g.,  if $\tau_1 = n$ for some $n > 1$, then $\omega_1$ is `frozen' and we smear in $\omega_{n + 1}$. 

Continuing: for each $i \geq 0$ and $0 \leq m < n$, the event $S_{i, m, n}$ (notation as in Lemma \ref{lem:correspond})
%\[
%S_{i, m, n} = \{ \tau_i = m\} \cap \{ \tau_{i + 1} \geq n\} \, ,
%\]
%each of which we may treat as a subset of 
can be treated as a subset of $[- \e , \e]^{n+1}$ since each $\tau_i$ is a stopping time w.r.t. $\Fc_n = \sigma(\omega_0, \cdots, \omega_n)$ (i.e., we have $\{ \tau_i > n \} \in \Fc_n$ for all $i, n$). Define as well the events $S_{i, n} = \{ \tau_i = n -1 \}$, and observe that the collection 
\[
\mathfrak P_n = \{ S_{i, n} : i \geq 1\} \cup \{ S_{i, m, n} : i \geq 1, 0 \leq m < n \}
\]
is a partition of $[- \e, \e]^{n+1}$. We define $\Ac_n \geq \mathfrak P_n$ on each $\mathfrak P_n$-atom separately.

%\smallskip

%Before proceeding, some additional notation: for $\hat X_0 \in S^1$ and $\hat \omega_i \in [- \e, \e], i \geq 2$, let us write $\hat \Qc_1[\hat X_0] \leq \hat \Qc_2[\hat X_0, \hat \omega_2] \leq \cdots \leq \hat \Qc_n[\hat X_0, \hat \omega_2, \cdots, \hat \omega_n] \leq \cdots$ for the \emph{itinerary partitions} $\Qc_1, \Qc_2, \cdots$ in Section 2.3 with the replacements $I = \hat X_0 + [- \e, \e]$, $\omega_1 = 0$ and $\omega_i = \hat \omega_i$ for $i >1$. We regard the $\hat \Qc_i$ as partitions of $[- \e, \e]$ in the obvious way.

\begin{itemize}
	\item For each set of the form $S_{i, m, n} \in \mathfrak P_n, i \geq 0, 0 \leq m < n$, we define $\Ac_n|_{S_{i, m, n}}$ to consist of atoms of the form
		\[
			\{ \omega_0 \} \times \{ \omega_1\} \times \cdots \times \{ \omega_{m} \} \times J \times \{ \omega_{m + 2}\} \times \cdots \times \{ \omega_n \} \, ,
		\]
		as $J$ ranges over the atoms of $\Qc_{n - m+1}(I_{m + 1}; (0, \omega_{m+2}, \cdots, \omega_n))$. Here we identify $[- \e, \e]$ with $I_{m + 1} = \tilde X_{m + 1} + [- \e, \e]$ in the obvious way.
	\item On each set $S_{i, n} \in \mathfrak P_n, i \geq 1$, we define $\Ac|_{S_{i, n}}$ to consist of atoms of the form
	\[
	\{ \omega_0\} \times \{ \omega_1\} \times \cdots \times \{ \omega_{n-1}\} \times [- \e, \e] \, .
	\]
\end{itemize}

With $\Ac_n$ completely described, the construction of $\Hc_n := \sigma(\Ac_n)$ is complete. It is not hard to check that
$\Hc_n$ is a filtration, i.e., $\Hc_n \supset \Hc_{n-1}$: to do this, one verifies that the partition sequence $\Ac_n$ is increasing
by inspecting each $\mathfrak B_n$-atom separately.

\bigskip

%The construction of the interval process $I_n$ is completely parallel to the itinerary 
%subdivisions $\Qc_1, \Qc_2, \cdots$ defined in Section 2.3. 
The following is a straightforward consequence of Lemma \ref{lem:correspond}.
\begin{lem} \label{lem:supportWorks}
	For each $n \geq 1$, the random interval $I_n$ is $\Hc_n$-measurable. Moreover, the measure $\nu_n (\cdot) = \P(\tilde X_n \in \cdot | \Hc_n)$ 
	satisfies $\operatorname{supp}(\nu_n) = I_n$.
\end{lem}

\subsection{The conditional measures $\nu_n$}

Let us first describe more transparently what the conditional measures $\nu_n(\cdot) = \P(\tilde X_n \in \cdot | \Hc_n)$ actually are. To start, for $\uo \in S_{i, n}, i \geq 0, n \geq 1$, we have that $\nu_n = \delta_{\tilde X_n} * \nu^\e$ is the uniform distribution on $I_n = \tilde X_n + [- \e, \e]$. The following characterizes $\nu_n$ on the event $S_{i, m, n}, i \geq 0, 0 \leq m < n$:

\begin{lem} 
Let $i \geq 0, 0 \leq m < n$ and condition on the event $S_{i, m, n} = \{ \tau_i = m, \tau_{i + 1} \geq n\}$. Define $\hat F_{m, n} : [-\e, \e] \to \R$ to be the map sending $\omega \mapsto \tilde X_n = \tilde f_{\omega_{n-1}} \circ \cdots \circ \tilde f_{\omega_{m + 2}} \circ \tilde f_\omega(\tX_{m+1})$. 

Let $J \in \Qc_{n - m-1}(\tilde X_{m+1}; (0, \omega_{m+2}, \cdots, \omega_n))$ (regarded as a partition of $[- \e, \e]$) be the atom containing $\omega_{m + 1}$. Then, 
$\hat F_{m, n} : J \to I_n$ is a diffeomorphism, and
\begin{align}\label{eq:altForm11}
\nu_n = \frac{1}{\nu^\e(J)} (\hat F_{m, n})_* (\nu^\e|_J) \, .
\end{align}

%Let $i \geq 1, 0 \leq m < n$. Define $\hat F_{m, n} : [-\e, \e] \to \R$ to map $\omega \mapsto \tilde f_{\omega_n} \circ \cdots \circ \tilde f_{\omega_{m + 2}} \circ \tilde f_\omega(\tX_m)$.
%Let $\uo \in S_{i, m, n}$, and let $J \in \Qc[(\tX_m, \ell_m), \omega_{m + 2}, \cdots, \omega_n]$ be the atom for which $\omega_{m + 1} \in J$. 
%Then, we have that $\hat F_{m, n}$ maps $J$ to $I_n$ diffeomorphically, and
%\begin{align}\label{eq:altForm11}
%\nu_n = \frac{1}{\nu^\e(J)} (\hat F_{m, n})_* (\nu^\e|_J) \, .
%\end{align}
\end{lem}

\noindent The proof is a case-by-case verification of the above formula and is left to the reader.

\medskip

Recall that $J \subset [- \e, \e]$ appearing in \eqref{eq:altForm11} has the property that points in $\tilde X_{m+1} + J$ have the same itinerary under $\tilde f^{n-m-1}_{\theta^{m + 2} \uo} \circ \tilde f$. In that notation, we have that the density $\rho_n = \frac{d \nu_n}{d \Leb}$ at a point $x \in I_n$ is, up to a constant scalar, given by
\[
(\hat F_{m, n})' (\omega) = (f^{n-m-1}_{\theta^{m+2} \uo} \circ f)'(\tilde X_{m+1} + \omega) %= f^{n-m}_{(\omega, \omega_{m+2}, \cdots)}(\tilde X_{m + 1}) \, ,
\]
where $\omega \in [- \e, \e]$ is such that $x = \hat F_{m, n} (\omega)$. In view of Proposition \ref{prop:distortion} and Lemma \ref{lem:correspond}, then, we obtain a distortion estimate for the density $\rho_n = \frac{d \nu_n}{d \Leb}$:

\begin{cor}\label{cor:distortion}
Let $n \geq 1$ be such that $I_n$ is free. Then, for all $x, x' \in I_n$, we have the estimate
\begin{align}
		\frac{\rho_n(x)}{\rho_n(x')} \leq \exp\big( K_2 L^{- 1/2} + 4 \| \psi'' \|_{C^0} L^{2 \b} |x - x'|\big)  \, .
\end{align}
\end{cor}

\section{Lyapunov exponents}

Finally, we come to the estimation of Lyapunov exponents in Theorem \ref{thm:lyapEst}.
 Throughout, we assume the setup of Section \ref{subsubsec:basicSetup} 
 and that $\e \geq L^{- (2 k + 1)(1 - \b) + \a}$ for some $\a \geq 0$. By Theorem \ref{thm:ergod},
 it follows that there is a unique ergodic stationary measure $\mu$ supported on $S^1$.

\medskip

 By (a version of) the Birkhoff ergodic theorem (see Corollary 2.2 on pg. 24 of \cite{kifer2012ergodic}), we
 have that
 \[
 \lambda = \lim_{n \to \infty} \frac1n \log | (f^n_\uo)'(x) | 
 \]
exists and is constant over $\P$-a.e. $\uo \in \Omega$ and $\mu$-a.e. $x \in S^1$. Since, 
however, $\mu$ is absolutely continuous and supported on all of $S^1$, we can promote
this limit to \emph{every} $x \in S^1$ and $\P$-a.e. $\uo \in \Omega$; details are left
to the reader.

%\begin{prop}
%Assume the setup in Section 2.1.1 and that $L$ is sufficiently large, depending only on $\b, c$ and $\psi$.
%Let $\mu$ be any ergodic stationary measure. Assume $\e \geq L^{- (2 k + 1)(1 - \b) + \a}$ 
%\end{prop}

%{\color{red} Need ergodicity}
%
%For the purposes of the following discussion, let $\e > 0$ and $\mu$ be a stationary 
%probability measure on $S^1$. By ergodicity, the Lyapunov 
%exponent $\lambda = \lim_{n \to \infty} \frac1n \log |(f^n_\uo)'(x)|$ exists and is constant over $\mu$-almost every $x \in S^1, \uo \in \Omega$. 

\medskip

It remains to estimate $\lambda$ from below, for which we use the following.
\begin{lem}
In the above setting, we have that
\[
\lambda \geq \inf_{x \in S^1} \liminf_{n \to \infty} \frac{1}{n} \E \big( \log |(f^n_\uo)'(x)| \big)
\]
for all $x \in S^1$.
\end{lem}
\begin{proof}
The limit
\[
\lambda = \lim_{n \to \infty} \frac1n \int_{S^1} \E \big( \log |(f^n_\uo)'(x)| \big) \, d \mu(x)
\]
follows from the $L^1$-Mean Ergodic Theorem applied to the skew product $\tau : S^1 \times \Omega \to S^1 \times \Omega$
defined by setting $\tau(x, \uo) = (f_{\omega_0} x, \theta \uo)$, on noting that $\mu$ is a stationary ergodic measure
iff $\mu \otimes \P$ is an ergodic invariant measure for $\tau$ (Theorem 2.1 on pg. 20 in \cite{kifer2012ergodic}). 

As is not hard to check, for all $x \in S^1$ we have $- d(L, \e) \leq \E \big( \log |f_\omega'(x) | \big) \leq \log L$
where $d(L, \e) > 0$ is a constant depending only on $\e, L$. These bounds pass to the averages $g_n := \frac1n \E \big( \log |(f^n_\uo)'(x)| \big)$. Applying Fatou's Lemma to the nonnegative sequence $g_n + d(L, \e)$, we conclude
\[
\inf_{x \in S^1} \liminf_{n \to \infty}  \frac{1}{n} \E \big( \log |(f^n_\uo)'(x)| \big)   \leq
\int_{S^1} \liminf_n g_n \, d \mu
\leq \lim_{n \to \infty} \int_{S^1} g_n \, d \mu = \lambda \, . \qedhere
\]

%\[
%\lambda = \int_{x \in S^1} \bigg( \liminf_{n \to \infty} \frac1n \E \big( \log |(f^n_\uo)'(x)| \big) \bigg) \, d \mu(x) 
%\]
%holds from the Bounded convergence theorem. 
\end{proof}

The remaining work is to estimate $\liminf_n \frac1n \E(\log |(f^n_\uo)'(x)|)$ for arbitrary $x \in S^1$. 
\begin{prop}\label{prop:LyapSuff}
For all $x \in S^1$, we have
\[
\liminf_{n \to \infty} \frac{1}{n} \E \big( \log  |(f^n_\uo)'(x)| \big) \geq \lambda_0 \log L \, ,
\]
where $\lambda_0 = \min \{ \frac{\a}{k + 1}, \frac{1}{10}\}$.
% = \gamma_0(\a, \b, k)$ is as defined in \eqref{eq:defineGamma0}.
\end{prop}

The proof of Proposition \ref{prop:LyapSuff} occupies the remainder of Section 6.

\medskip

{\it Reductions.} We make here some slight modifications to the upper and lower bounds on $\epsilon$ and the
parameter $\b$.
To start, on shrinking the parameter $\b$, we assume
\[
\e \geq L^{- \frac{1 - \b}{1 + \b} k  - (1 - \b) (k+1) + \a } \, .
\]
Second, we can assume without loss that $\e < L^{- \min \{ k-1, \frac12\}}$ as in the hypothesis
 \eqref{boundEpsilon} for Lemma \ref{lem:boundPeriodCoherence}. If not, then we can reduce to this case by 
 a similar line of reasoning as to the reductions in Section 3.2 in the proof of Theorem \ref{thm:ergod},
 to which we refer for details. 
 
 Finally, a minor technical point: we will assume that $k, \b$ satisfy the relation
 \begin{align}\label{eq:technicalKBound}
 	\bigg( \frac{3}{10} - \frac52 \b - \b^2 \bigg) k \geq 2 \b (1 + \b) \, .
 \end{align}
For $k \geq 6$, \eqref{eq:technicalKBound} is automatic for all $\b \in (0,1/10)$, and \eqref{eq:technicalKBound}
while holds for all $k \in \N$ when $\b \in (0,1/100)$. This entails no loss of generality.
 
 \bigskip
 
{\it With $\b$ fixed once and for all, we let $L$ be sufficiently large, in terms of $\b$, and
  take on the assumptions of Section \ref{subsubsec:basicSetup}. The parameter $\e$ is as above, and 
for our choice of $k \in \N$ we assume \eqref{eq:technicalKBound} holds. Finally, the
 constructions of Section 5 (namely, the filtration $\Hc_n$) are applied to the arbitrary initial condition $x = X_0 \in S^1$.}

%{\color{red} {\it Idea of the proof. }The main obstructions are visits to the `worst' bad set $\Bc^k$. 
%Using the selective averaging sigma-algebras $(\Hc_n)$, 
%we argue that when $X_n$ visits $\Bc^k$, we have that the
%law of $X_n$ conditioned on $\Hc_n$ is sufficiently `spread out'
%that the average of $\log |f'|$ is bounded from below (Proposition \ref{prop:main2}).
%}

%{\it Setup and reductions.} 
%\begin{itemize}
%{\color{red} 	\item $\b \in (0,1/10)$ and $c \in (0,1)$ fixed. $L$ always taken sufficiently large w.r.t. $\psi$ and $c,\beta$ alone.
%	\item $k \in \N$ arbitrary, $\e \geq L^{- (2 k + 1)(1 - \b)}$ assumed throughout.
%	\item WLOG assume $\e$ satisfies \eqref{boundEpsilon}. If not, change to $k'$ for which $L^{- (k' - 1)} > \e \geq L^{- k'}$
%	and work with $k$ replaced by $k'$.
%}\end{itemize}

\subsection{Decomposing the sum}

Fix $n \geq 1$. Define $T_i = \log |\tilde f_{\omega_{i }}' (X_i)|$, $X_i := f^i_\uo(x)$. With $\tau_0 \equiv 0 < \tau_1 < \tau_2 < \cdots$ as in Section 5, define the random index $J \in \Z_{\geq 0}$ to satisfy \[ \tau_J < n \leq \tau_{J + 1} \, ; \] note that $\tau_1 \geq n$ implies $J = 0$ since $\tau_0 := 0$.

We decompose
\[
(*): = \log|(f^n_\uo)'(x)| = \sum_{i = 0}^{n-1} T_i = \sum_{i =0}^{\min\{ \tau_1, n \} - 1} T_i  +
%\sum_{j = \tau_1}^{\min\{ \tau_2, n \} - 1} T_i \bigg) +  \tau
\sum_{j = 1}^\infty \chi_{J \geq j} \bigg( T_{\tau_j} + \sum_{i = \tau_j + 1}^{\min\{ \tau_{j+1}, n \} - 1} T_i \bigg) 
\]
and will bound $ \E (*)$ from below; here, for an event $A$ we write $\chi_A$ for the indicator function of $A$. The main obstacles are the terms $T_{\tau_j}, 1 \leq j \leq J$, which we bound from below using conditional expectations w.r.t. the filtration $(\Hc_n)_n$.
\begin{prop}\label{prop:main2}
Let $j \geq 2$ and condition on the event $ \tau_j = m$. Then,
\begin{align}
\E \big( T_{ m} | \Hc_{m }) \geq - \gamma \log L \, ,
%\min \big\{ \e L^{(k + 1)(1 - \b)}, L^{- \frac{k}{2} + 1 - 3 \b} \big\} =: - \gamma \log L \, .
\end{align}
where $\gamma :=  \max\{ (1 + \b) \big(  (\frac12 + \b) k+ 2 \b \big) , k (1 - \b) - \a \}$.
\end{prop}
\noindent Proposition \ref{prop:main2} is proved in Section 6.2.

\medskip

We apply Proposition \ref{prop:main2} by replacing the terms $\chi_{J \geq j} T_{ \tau_j}, j \geq 2$ under $\E$ with the conditional expectations\footnote{For a filtration $(\Gc_n)$ and an adapted stopping time $\eta$, we write $\Gc_\eta$ for the \emph{stopped $\sigma$-algebra} consisting of the set of measurable sets $A$ for which $A \cap \{ \eta \leq m \} \in \Gc_m$ for all $m$.}
\[
(*)_j := \E \big( \chi_{J \geq j} T_{\tau_j} | \Hc_{\tau_j} \big) =  \sum_{m = 1}^{n-1} \E \big( \chi_{\tau_j = m} T_m | \Hc_{m } \big) = \sum_{m = 1}^{n-1} \chi_{ \tau_j = m} \cdot \E \big(  T_m | \Hc_{m } \big) \, .
\]
Here, we use that $\{ J \geq j \} = \cup_{m = 1}^{n-1} \{ \tau_j = m \}$ for all $j \geq 1$. By Proposition \ref{prop:main2}, for $j \geq 2$ we have
\[
(*)_j \geq - \gamma \log L \cdot \chi_{J \geq j} \, .
\]

%
%
%For $j \geq 2$ we have
%\[
%(*)_j := \E \big( \chi_{J \geq j} T_{\hat \tau_j} \big) = \sum_{m = 1}^{n-1} \E \big( \chi_{J \geq j, \hat \tau_j = m} T_{m} \big) = \sum_{m = 1}^{n-1} \E \bigg( \E \big( \chi_{J \geq j, \hat \tau_j = m} T_{m} | \Hc_{m + 1} \big) \bigg) 
%\]
%It is clear that $\{ J \geq j, \hat \tau_j = m\} = \{ J \geq j , \tau_j = m + 1\} = \{ \tau_j = m + 1\}$ is $\Hc_{m + 1}$-measurable (indeed, $\{ J \geq j\} = \cup_{m = 1}^{n-1} \{ \hat \tau_j = m\}$), and so can be pulled out of the conditional expectation with respect to $\Hc_{m + 1}$. Applying Proposition \ref{prop:main2}, we have
%\[
%(*)_j \geq - \gamma \log L \cdot \sum_{m = 1}^{n-1} \E \big( \chi_{J \geq j,\hat  \tau_j = m } \big) = - \gamma \log L \cdot \sum_{m = 1}^{n-1} \P(\hat \tau_j = m) = - \gamma \log L \cdot \E \big( \chi_{J \geq j} \big)
%\]

For the $j = 1$ term, we use the following crude estimate:
\begin{lem}\label{lem:crudeBound}
We have
\[
(*)_1 := \E \big( \chi_{J \geq 1} T_{ \tau_1} | \Hc_{\tau_1} \big) \geq - 2 (2 k + 1) \log L  =: - \gamma_1 \log L \, .
\]

\end{lem}
We prove Lemma \ref{lem:crudeBound} in Section 6.2.

Applying these estimates, we have
\begin{align*}
\E(*) & \geq \E \bigg[  \sum_{i =0}^{\min\{  \tau_1, n \} - 1} T_i + (*)_1 + \chi_{J \geq 1} \sum_{i =  \tau_1 + 1}^{\min\{  \tau_2 , n \} - 1} T_i + \sum_{j = 2}^\infty \bigg( (*)_j  +  \chi_{J \geq j} \sum_{i =  \tau_j + 1}^{\min\{  \tau_{j+1}, n \} - 1} T_i  \bigg)  \bigg] \\
& \geq \E \bigg[ \underbrace{ \sum_{i =0}^{\min\{  \tau_1, n \} - 1} T_i }_{I} + 
\underbrace{\chi_{J \geq 1} \bigg( - \gamma_1 \log L + \sum_{i =  \tau_1 + 1}^{\min\{  \tau_2 , n \} - 1} T_i \bigg) }_{II}
  + \underbrace{ \sum_{j = 2}^\infty \chi_{J \geq j} \cdot \bigg( - \gamma \log L  +   \sum_{i =  \tau_j + 1}^{\min\{  \tau_{j+1}, n \} - 1} T_i  \bigg) }_{III} \bigg] \\
& =: \E[I + II + III] \, .
\end{align*}
To complete the estimate, we decompose according to the events $\{J = K\}, K = 0,1,2,\cdots$. 

\bigskip

\noindent {\bf (A) Estimate of $\E \big( \chi_{J = 0} (I + II + III) \big)$. }

We have $II = III = 0$ and
\[
\E [\chi_{J = 0} \cdot I ] = \E \bigg[ \chi_{J = 0}  \sum_{i = 0}^{n-1} T_i  \bigg] 
\]
Conditioned on $J = 0$, we have $\tau_1 \geq n$ and so Lemma \ref{lem:fullBoundPer}
may be applied (see also Lemma \ref{lem:correspond}). We obtain a lower bound using the worst possible case that $p_{\omega_{n-1}}(I_{n-1}) = k-1$, i.e., $I_{n-1}$ initiates a bound period of length $k-1$ at time $n-1$ (corresponding to $t_j = n-1, p_j = k-1$ In the notation of Lemma \ref{lem:fullBoundPer}(b)). So,
\[
\sum_{i = 0}^{n-1} T_i = \log |(f^n_\uo)'(x_0)| \geq L^{(n-1)(\frac12 - \b) - \frac{k-1}{2} - \b} \, .
\]
We conclude
\begin{align*}
\E[ \chi_{J = 0} \cdot I ] \geq \bigg(  (n-1) \big( \frac12 - \b \big) \log L -( \frac{k-1}{2} + \b) \log L \bigg) \cdot \P (J = 0) \, .
\end{align*}

%The period from $0$ to $n-1$ consists of completed bound periods of length $\leq k - 1$ and at most one possibly incomplete bound period at the end initiated from $\Bc_{\omega_i}^{l}, $ for some $1 \leq l^* \leq k -1$ and $i^* \leq n-1$; that is, $\check I_{i^*} \cap \Bc^{l^*}_{\omega_{i^* + 1}} \neq \emptyset$ and {\color{red} $i^* + k \geq n -1$}. So, conditioned on $\{ J = 0 \}$, we have

%\begin{align*}
%\E \bigg[ \sum_{i = 0}^{n-1} T_i \bigg]  & \geq \E \bigg[ \chi_{J = 0} \bigg(  i^* \cdot \big( \frac12 - \b \big) \log L + T_{i^*} + \sum_{i = i^* + 1}^{n-1} T_i \bigg) \bigg] \\
%& \geq \bigg( i^* \cdot \big( \frac12 - \b \big) \log L + (1 - \frac{l^* + 1}{2} - \b) \log L + (n - i^* - 1) \big( \frac12 - \b \big) \log L \bigg) \cdot \P (J = 0) \\
%& \geq \bigg(  (n-1) \big( \frac12 - \b \big) \log L + (1 - \frac{k}{2} - \b) \log L \bigg) \cdot \P (J = 0) \, .
%\end{align*}

\bigskip

\noindent {\bf (B) Estimate of $\E \big( \chi_{J = 1} (I + II + III) \big)$. }

Here we have $III = 0$ and
\[
\E [ \chi_{J = 1} \cdot (I + II) ] = \E \bigg[ \chi_{J = 1} \bigg( \sum_{i = 0}^{ \tau_1 - 1} T_i - \gamma_1 \log L + \sum_{i =  \tau_1 + 1}^{n-1} T_i \bigg)  \bigg]
\]
By Lemma \ref{lem:fullBoundPer}(a) we have $\sum_{i = 0}^{\tau_1 - 1} T_i \geq  \tau_1 \cdot \big( \frac12 - \b \big) \log L$. The second summation $\sum_{i =  \tau_1 + 1}^{n-1} T_i$ is estimated as in paragraph (A): we have
\[
\E \bigg[ \chi_{J = 1} \sum_{i =  \tau_1 + 1}^{n-1} T_i \bigg] \geq \bigg(  (n - 2 -  \tau_1) \big( \frac12 - \b \big) \log L - (\frac{k-1}{2} + \b) \log L \bigg) \cdot \P(J = 1) \, ,
\]
and so collecting, we get
\[
\E [ \chi_{J = 1} \cdot (I + II) ] \geq \bigg( (n - 2) \big( \frac12 - \b \big) \log L - \gamma_1 \log L + \big( 1 - \frac{k}{2} - \b \big) \log L \bigg) \cdot \P(J = 1) \, .
\]

\bigskip

\noindent {\bf (C) Estimate of $\E \big( \chi_{J = K} (I + II + III) \big)$ for $K > 1$.}

We bound $\E[\chi_{J = K} \cdot (I + II)]$ as in paragraph (A), obtaining
\[
\E [\chi_{J = K} \cdot (I + II) ] \geq \E \bigg[ \chi_{J = K} \bigg(  (\tau_2 - 1) \big( \frac12 - \b \big) \log L - \gamma_1 \log L \bigg) \bigg] \, .
\]
%{\color{green}Add $(-\frac{k}{2}-\beta)\log L$?}

%For the $II$ term, the period from $\hat \tau_1 + 1$ to $\hat \tau_2 - 1$ likewise has no incomplete bound periods, hence
%\[
%\E [\chi_{J = K} \cdot II ] \geq \E \bigg[ \chi_{J = K} \cdot \bigg( (\hat \tau_2 - \hat \tau_1 - 1)  \big( \frac12 - \b \big) \log L - \gamma_1 \log L \bigg) \bigg] \, .
%\]

Conditioned on $\{ J = K\}$ for $K > 1$, the $III$ term has the form
\[
III = \sum_{j = 2}^{K-1} \underbrace{\bigg( - \gamma \log L + \sum_{i =  \tau_j + 1}^{ \tau_{j + 1} - 1} T_i \bigg) }_{IV_j}  
+ \bigg( \underbrace{- \gamma \log L + \sum_{i =  \tau_K + 1}^{n-1} T_i}_{IV_K} \bigg) 
\]
For each summand $IV_j, j \geq 2$, observe that $\tilde X_i \in \Gc$ for each $i = \tau_j + 1, \cdots, \tau_j + k$, hence $\sum_{i = \tau_j + 1}^{\tau_j + k} T_i \geq k (1 - \b) \log L$. If $\tau_j +k + 1 \leq \tau_{j+1} - 1$, then the summands $\tau_j +k + 1 \leq i
\leq  \tau_{j+1} - 1$ are estimated as in Lemma \ref{lem:fullBoundPer}(a). In total,
\[
\E [ \chi_{J = K} \cdot IV_j ] \geq \E \bigg[ \chi_{J = K} \bigg( \big( k (1 - \b) - \gamma \big) \cdot \log L + (  \tau_{j + 1} - 1 -  \tau_j - k) \cdot \big( \frac12 - \b) \log L 
\bigg) \bigg]
\]
Observe that 
\begin{align*}
k (1 - \b) - \gamma &=  \min \{  \a , \big( \frac12 - \frac52 \b - \b^2 \big) k - 2\b (1 + \b)\} \geq \min \{ \a (k + 1), \frac15 k \}
\end{align*}
holds from \eqref{eq:technicalKBound}. Dividing the latter by $k + 1$
%\min \{ (1 - \b)(2 k + 1) - (1 + \b) \a , \frac{k}{2} (1 - \b)\}$
 yields an estimate for the average growth rate $\lambda_0$ as follows:
\begin{align}\label{eq:defineGamma0}
\frac{k (1 - \b) - \gamma}{k + 1} \geq \min \{ \a, \frac{1}{10}\} =: \lambda_0 = \lambda_0(\a, k) \, ,
% := \min \{ \frac{\a}{k + 1} , \frac{1}{10} \} \, ,
\end{align}
hence
\[
\E [ \chi_{J = K} \cdot IV_j ] \geq \E \big[ \chi_{J = K} ( \tau_{j + 1} -  \tau_j) \cdot \lambda_0 \log L \big] \, .
\]
This telescopes, and so 
\[
\E [\chi_{J = K} \sum_{j = 2}^{K-1} IV_j ] \geq \E \big[ \chi_{J = K} (  \tau_K -  \tau_2) \cdot \lambda_0 \log L \big] 
\]
Using Lemma \ref{lem:fullBoundPer}(b) we bound $IV_K$ from below by
\[
IV_K = - \gamma \log L + \sum_{j = \tau_K + 1}^{n-1} T_i \geq - \gamma \log L + (n - \tau_K - 2) (\frac12 - \b) \log L 
-( \frac{k-1}{2} + \b) \log L
\] 
hence
\[
\E [ \chi_{J = K} \cdot III] \geq \E \bigg[ \chi_{J = K} \bigg(
(n-2 -  \tau_2 ) \cdot \lambda_0 \log L - \gamma \log L - (\frac{k-1}{2} + \b) \log L
 \bigg) \bigg]
\]
and in total,
\[
\E[ \chi_{J = K} (I + II + III)] \geq \bigg( (n-3) \lambda_0 \log L - (\gamma + \gamma_1) \log L - (\frac{k-1}{2} + \b) \log L \bigg) \cdot \P(J = K) \, .
\]

\bigskip
\noindent {\bf Putting it together.}

\medskip

The lower bounds obtained for $K > 1$ as in paragraph (C) are the worst of the three cases examined already, hence
\[
\E (*) = \E[I + II + III] = \sum_{K = 0}^\infty \E[ \chi_{J = K} (I + II + III)] \geq (n-3) \lambda_0 \log L - (\gamma + \gamma_1) \log L - (\frac{k-1}{2} + \b) \log L \, .
\]
On dividing by $n$ and taking $n \to \infty$, we conclude that 
$
\lim_{n \to \infty} \frac1n \E \big( \log|(f^n_\uo)'(x)| \big) \geq \lambda_0 \log L \, ,
$
as desired.

%\[
%\E \chi_{J = K} (**) = \E \bigg[ \sum_{i = 0}^{\hat \tau_1 - 1} T_i - \gamma_1 \log L +\sum_{i = \hat \tau_1 + 1}^{\hat \tau_2 - 1} T_i +  \sum_{j = 2}^{K-1} \bigg( - \gamma \log L + \sum_{i = \hat \tau_j + 1}^{\hat \tau_j + k} T_i + \sum_{i = \hat \tau_j + k + 1}^{\hat \tau_{j + 1} - 1} T_i \bigg) - \gamma \log L + \sum_{i = \hat \tau_J + 1}^{n-1} T_i\bigg]
%\]
\subsection{Proofs of Proposition \ref{prop:main2} and Lemma \ref{lem:crudeBound}}

Below, $C > 0$ refers to a constant depending only on $\psi$, and may change in value from line to line.

\medskip

We start with the following preliminary estimate. 

\begin{lem}\label{lem:boundIntegralBelow11}
Let $I \subset \Bc$ be any connected interval. Then,
\[
\int_I \log | f'(z) | \, d z \geq |I| \cdot \log (L^{1 - \b} |I|) \, .
\]
\end{lem}
This is a simple consequence of \eqref{eq:lowerBoundDer} and follows on taking $L$ sufficiently large, depending only on $\b$ and $\psi$; details are left to the reader. 

\begin{proof}[Proof of Proposition \ref{prop:main2}]
Unconditionally, for any $m \geq 0$ the conditional expectation $\E(T_m | \Hc_m)$ is given by
\[
(**) = \int_{I_m} \log |f'_{\omega_m}(z) | \, d \nu_m(z) \, .
\]
by Lemma \ref{lem:supportWorks}.

Conditioning on $\{ \tau_j = m\}$, recall (Remark \ref{rmk:smallAtoms}) that $|I_m| \leq C L^{- \frac{k}{2} - \b}$ since $I_m$ is an atom of $\Pc_{\omega_m}(\tilde f_{\omega_{m-1}}(I_{m-1}))$. Our distortion control on $\rho_m = \frac{d \nu_m}{d \Leb}$ as in Corollary \ref{cor:distortion} along $I_m$ implies $| \log \frac{\rho_m(z)}{\rho_m(z')} | \leq K_2 L^{- 1/2} + 2 K_1^{-1} L^{- \frac{k}{2} + \b} \leq C L^{- 1/2 + \b}$ for $z, z' \in I_m$, hence
\[
(**) \geq (1 + C L^{-1/2 + \b} ) \frac{1}{|I_m|} \bigg( \int_{I_m} \log | f'_{\omega_{m+1}} (z) | \, d z \bigg) \, .
\]
%For the parenthetical integral, we have the following simple bound.

%\medskip

From Lemma \ref{lem:boundIntegralBelow11} applied to $I = I_m$, we conclude
\begin{align}\label{eq:boundImBelow11}
(**) \geq (1 + C L^{-1/2 + \b}) \log (L^{1 - \b} |I_m|) \geq  (1 + \b) \log (L^{1 - \b} |I_m|) \, .
\end{align}

We now bound $|I_m|$ from below.

\begin{lem}\label{lem:boundIm}
On the event $\{ \tau_j = m\}, j \geq 2, m \geq 1$, we have the estimate \[|I_m| \geq \min\{ L^{-1 - (\frac12 + \b) k - \b}, L^{k (1 - \b)} \e\} \, .\]
\end{lem}

Assuming this and plugging in $\e \geq L^{- \frac{1 - \b}{1 + \b} k  - (1 - \b) (k+1) + \a }$, we conclude
\begin{align*}
(**) & \geq (1 + \b) \log \min\{ L^{ - (\frac12 + \b) k -2 \b} , L^{(k+1) (1 - \b)} \e \} \\
& \geq \min \{ (1 + \b) \big( - 2 \b - (\frac12 + \b) \big) , (1 + \b) \big( \a - k \frac{1 - \b}{1 + \b} \big) \} \\
& \geq - \max\{ (1 + \b) \big(  (\frac12 + \b) k+ 2 \b \big) , k (1 - \b) - \a \} \log L =: - \gamma \log L \, .
\end{align*}

To finish the proof of Proposition \ref{prop:main2}, it remains to prove Lemma \ref{lem:boundIm}.
\end{proof}

\begin{proof}[Proof of Lemma \ref{lem:boundIm}]
We distinguish two cases:
\begin{itemize}
	\item[(a)] $I_{i} = f_{\omega_{i-1}}(I_{i-1})$ for each $\tau_{j-1} + 2 \leq i \leq m = \tau_j$
	\item[(b)] $I_i \subsetneq f_{\omega_{i-1}}(I_{i-1})$ for some $\tau_{j-1}+2 \leq i \leq m = \tau_j$.
\end{itemize}
In case (a), we easily have $|I_{\tau_{j-1} + k + 1}| \geq L^{k (1 - \b)} \e$, and since no additional cuts are made, we estimate
\begin{align*}
| I_{m}| 
%&= | \tilde f_{\omega_{m-1 }} \circ \cdots \circ \tilde f_{\omega_{ \tau_{j-1} + 2}} \circ \tilde f (\tilde X_{  \tau_{j-1} +1} + [- \e, \e])|  \\ 
& = | \tilde f_{\omega_{m-1}} \circ \cdots \circ \tilde f_{\omega_{\tau_{j-1} + k + 1}} (I_{\tau_{j-1} + k + 1} ) | \\
& \geq L^{( m - ( \tau_{j-1} + k + 1))(\frac12 - \b)} |I_{\tau_{j-1} + k + 1}| \geq L^{k (1 - \b)} \e \, .
\end{align*}

In case (b), set $i^* = \max \{ i \leq  \tau_j : I_i \subsetneq \tilde f_{\omega_{i-1}}(I_{i-1}) \}$ (note $i^* = m$ is possible), and note that if $i^* < m$ then
\[
I_m = \tilde f_{\omega_{m-1}} \circ \cdots \circ \tilde f_{\omega_{i^*}}(I_{i^*}) \, .
\]
To bound $|I_{i^*}|$ we split further to the cases (i) $p_{\omega_{i^*}}(I_{i^*}) = 0$, (ii) $p_{\omega_{i^*}}(I_{i^*}) \in \{ 1,\cdots, k-1\}$ and (iii) $p_{\omega_{i^*}}(I_{i^*}) = k$. Note that in all cases, $\Pc_{\omega_{i^*}}(\tilde f_{\omega_{{i^*}-1}}(I_{{i^*}-1}))$ contains at least two elements, hence $I_{i^*}$ contains at least one atom of $\Pc_{\omega_{i^*}}$ (Remark \ref{rmk:smallAtoms}).

In case (b)(i), $I_{i^*} \subset \Ic_{\omega_{i^*}} \cup \Gc_{\omega_{i^*}}$. Either $I_{i^*}$ contains an atom of $\Gc_{\omega_{i^*}}$, in which case $|I_{i^*}|$ is bounded from below by $\frac12 \min\{ d(\hat x, \hat x') : \hat x, \hat x' \in C_\psi' , \hat x \neq \hat x'\}$, or $I_{i^*}$ contains an atom of $\Pc_{\omega_{i^*}}|_{\Ic_{\omega_{i^*}}}$, hence $|I_{i^*}| \geq L^{- \frac{3}{2} - \b}$ (the latter bound being the worse of the two). Since $I_m = I_{\tau_j}$ is free, we conclude $|I_m| \geq |I_{i^*}| \geq L^{- \frac32 - \b}$ from Lemma \ref{lem:fullBoundPer}(a).

In case (b)(ii), we have automatically that $I_{i^*}$ is free and initiates a bound period of length $p^* = p_{\omega_{i^*}}(I_{i^*})$.
Since $0 < p^* < k-1$ by assumption, we cannot have $i^* = \tau_j = m$ (since then $p^* = k$) and so conclude $i^* < \tau_j$ in this case-- indeed, we have $i^* + p^* + 1 \leq m = \tau_j$, since $I_{\tau_j}$ is free. From Remark \ref{rmk:smallAtoms} we have 
\[
| I_{i^*} | \geq (p^* + 1)^{-2} L^{- \frac{p^* + 3}{2} - \b} \geq L^{- \frac{p^* + 3}{2} - \b (p^* + 1)  } \, ,
\]
on taking $L$ large enough so $\b > 2/\log L$. Moreover, since $I_{m} = I_{\tau_j}$ is free, we have
\[
|I_m| \geq |I_{i^* + p^* + 1}| = |\tilde f^{p^*  +1}_{\theta^{i^*} \uo} (I_{i^*})| \geq L^{(p^* + 1)(\frac12 - \b)} |I_{i^*}| 
\geq L^{(p^* + 1) (\frac12 - \b) } \cdot L^{- \frac{p^* + 3}{2} - \b(  p^* + 1) } = L^{-1 - 2 \b (p^* + 1)}
\]
The worst possible case is $p^* = k-1$, and so we conclude $|I_m| \geq L^{-1 - 2 \b k}$ in case (ii).

In case (b)(iii), we have necessarily that $i^* = m = \tau_j$. In the worst case, $I_m$ contains an atom of $\Pc_{\omega_m}|_{\Bc^{k-1}_{\omega_m}}$, and so $|I_m| \geq k^{-2} L^{- \frac{k + 2}{2} - \b} \geq L^{- 1 - (\frac12 + \b) k - \b}$. 
\end{proof}

\bigskip

\begin{proof}[Proof of Lemma \ref{lem:crudeBound}]

Arguing in parallel to the proof of Proposition \ref{prop:main2} (see \eqref{eq:boundImBelow11}) we have, on the
event $\{ \tau_1 = m\}$, the estimate
\[
\E ( T_{m} | \Hc_{m}) \geq (1 + \b)  \log (L^{1 - \b} |I_m|)
\]
As before, we estimate $|I_m|$ from below.

\begin{lem}\label{lem:boundImTake2}
On the event $\{ \tau_1 = m \}$, we have the estimate $|I_m| \geq \min \{ L^{-1 - ( \frac12 + \b) k - \b}, \e\}$.
\end{lem}
Assuming this, we easily obtain
\[
\E(T_m | \Hc_m) \geq (1 + \b) \log (L^{1 - \b} \min \{ L^{-1 - ( \frac12 + \b) k - \b}, \e\} ) \geq - 2 (2 k + 1) \log L \, ,
\]
as claimed. It remains to prove Lemma \ref{lem:boundImTake2}.
\end{proof}

\begin{proof}[Proof of Lemma \ref{lem:boundImTake2}]
Condition on $\tau_1 = m$. 
The proof is very much parallel to that of Lemma \ref{lem:boundIm}. Case (b) 
can be repeated verbatim, and
yields the identical estimate $|I_m| \geq L^{-1 - (\frac12 + \b) k - \b}$.

The only difference is in case (a). Here, we observe that $I_m$ must be free,
and so (Lemma \ref{lem:fullBoundPer}(a)) we have
\[
|I_m| \geq L^{m (\frac12 - \b)} \cdot 2 \e \geq \e \, .
\]
This completes the proof of Lemma \ref{lem:boundImTake2}.
\end{proof}

%{\color{red} *** Proof of Lemma \ref{lem:crudeBound} ***}

%{\color{red} Let us now dispense with Lemma \ref{lem:crudeBound}: we have
%\[
%\E \big( \chi_{J \geq 1} T_{\tau_1} \big) = \sum_{m = 1}^{n-1} \E \big( \chi_{\tau_1 = m} T_m \big) \, .
%\]
%}
%
%Estimate using conditional expectation w.r.t. $\Hc_{\tau_1}$ instead... only way to make sense of pulling out
%the event $\{ J \geq 1\}$.

\bibliography{biblio}
\bibliographystyle{plain}

\end{document}